\documentclass[preprint,11pt]{elsarticle}

\usepackage{amsmath,amssymb,amsfonts, amsthm}
\usepackage[colorlinks]{hyperref}
\usepackage[capitalize,nameinlink,noabbrev]{cleveref} 

\usepackage{color}

\usepackage[center,font=small]{caption}
\usepackage{subfigure}

\newtheorem{theorem}{Theorem}[section]
\newtheorem{lemma}[theorem]{Lemma}

\newtheorem{remark}[theorem]{Remark}

\newtheorem{example}{Example}[section]

\usepackage{hyperref}

\journal{(Under review)}

\begin{document}
\begin{frontmatter}
\title{A finite difference scheme for two-dimensional singularly perturbed convection-diffusion problem with discontinuous source term}

\author[NITT,Woxen]{Ram Shiromani}\ead{ram.panday.786@gmail.com}
\author[UG]{Niall Madden\corref{cor1}}\ead{Niall.Madden@UniversityOfGalway.ie}
\author[NITT]{V.~Shanthi}\ead{vshanthi@nitt.edu}

\cortext[cor1]{Corresponding author}
\affiliation[NITT]{department={Department of Mathematics,},
    organization={National Institute of Technology Tiruchirappalli}, state={Tamil Nadu}, country={India}}
             
\affiliation[Woxen]{department={School of Sciences,},  organization={Woxsen University},
             addressline={Sadasivpet, Sangareddy District},
             city={Hyderabad},  country={India}}

\affiliation[UG]{department={School of Mathematical and Statistical Sciences,},  organization={University of Galway}, country={Ireland}}

\begin{abstract}
We propose a finite difference scheme for the numerical solution of a two-dimensional singularly perturbed convection-diffusion
partial differential equation whose solution features interacting boundary and interior layers, the latter due to discontinuities in source term.
The problem is posed on the unit square. The second derivative is multiplied by a singular perturbation parameter, $\epsilon$, while 
the nature of the first derivative term is such that flow is aligned with a boundary. 
These two facts mean that solutions tend to exhibit layers of both exponential and characteristic type.
We solve the problem using a finite difference method, specially adapted to the discontinuities, and applied on a  
piecewise-uniform (Shishkin). We prove that that the computed solution converges to the true one at a rate that is independent of the perturbation parameter, and is nearly first-order. We present numerical results that verify
that these results are sharp.
\end{abstract}
	

\end{frontmatter}	
		
\section{Introduction}\label{sec:intro}
We are interested in the approximation of solutions to singularly perturbed differential equations (SPDEs) 
by finite differences methods (FDMs). The partial differential equation in question is singularly perturbed in the sense that it
features a small parameter scaling the highest derivative; the perturbation is 
``singular'' in the sense that if it is formally set to zero, the problem becomes ill-posed. One consequence of this is that it is difficult to 
obtain a uniformly correct asymptotic expansion.  Moreover, solutions to such problems usually exhibit boundary layers: regions near the boundary of 
the domain where the solution's  gradient is extremely large. The equations we study are of convection-diffusion type, which are both 	
the most commonly studied singular perturbation problems in the literature, and the most challenging.  
If one formally sets $\epsilon=0$, then the characteristics of the reduced equation are parallel to the boundary, and the associated 
layers are of \emph{characteristic} (or \emph{parabolic}) type; otherwise they are of \emph{exponential} (or regular) type.
The problem we study is further complicated by the presence of discontinuities in the source terms, and so layers may feature in the interior of the domain.
  
SPDEs are ubiquitous in mathematical modelling. The problems in this paper are of \emph{convection-diffusion} type; Morton  lists a range of
examples of modules featuring such problems, including pollutant dispersal in rivers and in the atmosphere, 
semi-conductor simulations, and computational finance~\cite{Morton96}.
In all cases, once the model reaches any degree of sophistication, the SPDEs are too complicated to solve exactly, and so computational
methods are required. However classical techniques are highly sub-optimal (or, indeed, can fail entirely) for SPDEs (see, e.g.,~\cite{RST08,Stynes18}), and so this is a very active area of research. The ultimate goal is to devise a computational scheme that can produce approximations which resolve any layers present, and for which one can prove an error bound that is not dependent on the perturbation parameter. Such methods are called	\emph{parameter robust}.
  
The literature on parameter robust methods for SPDEs is extensive (see, e.g.,~\cite{RST08}, and references therein).		
Here, we examine those that are most pertinent to this study. The starting point is often the derivation of bounds on derivatives of solutions. Lin{\ss} and Stynes~\cite{linss2001asymptotic} consider the case where the solution of the convection-diffusion problem features (only) exponential layers. O'Riordan and Shishkin~\cite{o2008parameter} provide  careful analysis of  version of this problem, with variable
coefficients, where characteristic layers can be present (see also \cite{shagi1987asymptotic, roos2002optimal}). In those studies, particular emphasis is on the derivation of solution decomposition, which are key in establishing robust convergence of methods applied on the celebrated, widely-studied \emph{Shishkin} mesh. Such approaches have proved to be broadly applicable: at the time of writing, MathSciNet reports nearly 300 papers that feature Shishkin meshes in the previous 5 years\footnote{Searched on 4 Jan 2024, with search term \texttt{any:(Shishkin mesh) y:[2019 2023]}.}, and Google Scholar about 2,700.


In this study, we present a FDM for a convection-diffusion problem posed on the unit square, and where convection is parallel to a boundary. Consequently, the solution may exhibit both exponential and characteristic boundary layers. Moreover, the source term may be discontinuous across two intersecting interior lines, each parallel to a boundary. The solution tends to have quite intricate structure, with interacting interior layers that are themselves, exponential or characteristic in nature. Such problems are important in applications. For example, \cite{MaMo11}, where a solute transport problem, with a single discontinuity in the forcing term, is solved numerically using a FDM on a Shishkin-type mesh. However, its construction is based purely on heuristics, rather than the theoretical approach given here. (See also \cite{PoMo21} where the same approach is applied  to modelling flow in a catheterized artery).

We formulate the SPDE as 
\begin{subequations}
  \label{eq:full problem}
  \begin{equation}\label{eqncon}
    L_{\epsilon}u = f(x,y), \quad \forall(x,y)\in \Omega, \qquad u = q, \quad \forall(x,y) \in \partial \Omega,
  \end{equation}
  where the differential operator is 
  \begin{equation}
    \label{contoperator}
    L_{\epsilon}u = -\epsilon^2\bigg(\dfrac{\partial^2 u}{\partial x^2}
    + \dfrac{\partial^2 u}{\partial y^2}\bigg)+a(x,y)\dfrac{\partial
      u}{\partial x} + b(x,y) u.
  \end{equation}
\end{subequations}
The perturbation parameter, $\epsilon$, is positive but may be arbitrarily small, so we consider the scenario where  $0 \ll \epsilon <1$. The problem's  domain  is
$\Omega=\cup_{k=1}^{4}\Omega_k$, where the quadrants are defined as
\begin{equation}\label{eq:subdomains}
  \begin{split}
    \Omega_1=(0,d_1)\times(0,d_2), \quad & \quad  \Omega_2=(d_1,1)\times(0,d_2),\\
    \Omega_3=(0,d_1)\times(d_2,1), \quad & \quad \Omega_4=(d_1,1)\times(d_2,1),
  \end{split}
\end{equation}
and $d_1$ and $d_2$ are any point in $(0,1)$, providing they are sufficiently far from the
boundary so as not the be within a boundary layer region; see Remark~\ref{rem:d big enough} for a more detailed discussion.
\cref{fig:Omega} shows these four subregions, as well as further
subregions used later in the analysis. \begin{figure}
  \centering
  \includegraphics[width=12cm]{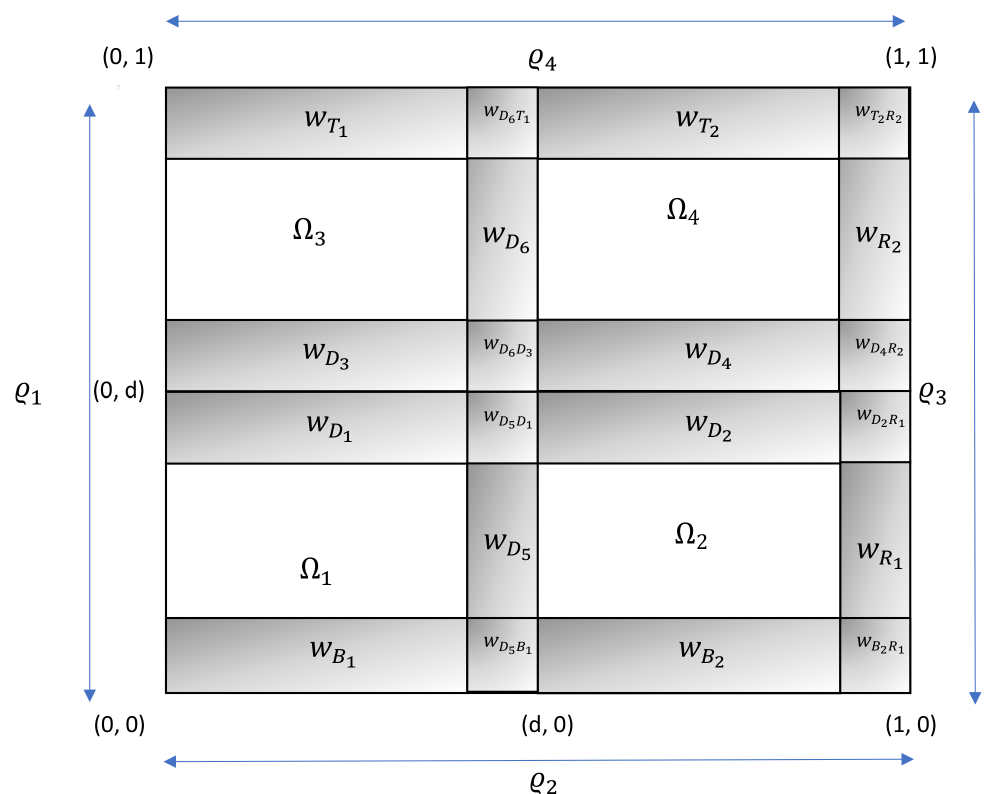}
  \caption{Notation for subregions and boundaries the domains}%
  \label{fig:Omega}
\end{figure}

The convection and reaction coefficients are positive, and bounded below as
\begin{equation}\label{eq:a b bounds}
  a(x,y) \geq \alpha >0,
  \qquad \text{ and } \qquad 
  b(x,y) \geq \beta^2 >0,
\end{equation}
for some constants $\alpha$ and $\beta$.
We shall assume that $b \in C^{4,\zeta}(\bar{\Omega})$ and $q  \in
C^{4,\zeta}({\partial \Omega})$, for some $\zeta \in (0,1]$.
Further, we assume the problem data is such that  $u \in
C^{1,1}(\bar{\Omega}) \cap C^{2,2}({\Omega})$
(see~\cite{han1990differentiability,liu2009two}).

The key distinction in \eqref{eq:full problem}, compared to other papers in the literature, is that 
the source term $f$ has a jump discontinuity at both lines $x=d_1$ and
$y=d_2$. We denote such a jump in a function $v$ at a point $(x,y) \in \Omega$ along the lines parallel to
the $x$- and $y$-axes as $[v](d_1,y)=v({d_1}^+,y) - v({d_1}^-,y)$ and
$[v](x,d_2)=v(x,{d_2}^+) - v(x,{d_2}^-)$ respectively.
		
The FDM that we propose is based on standard upwinding, but specially adjusted at the discontinuities. The scheme is applied on a
specialised piecewise uniform layer-adapted mesh of Shishkin type. At a glance, this is similar to the approach taken in \cite{RaShPr2023}, where a convection-diffusion problem with interior layers (due to discontinues in the forcing and convective coefficients) is also resolved using a FDM on a piecewise uniform mesh. However, the solution to the problem considered there features only exponential layers (since the convection is not aligned with the boundaries), whereas we consider a problem with both characteristic and exponential layers.

The rest of this manuscript is arranged as follows. Section~\ref{sec:a priori} considers the continuous problem: we derive a maximum principle and 
stability result for the operator, and give sharp (with respect to $\epsilon$) 
bounds of the solution and its derivatives.
The FDM is presented in Section~\ref{sec:FDM}, along with a suitable Shishkin mesh. In Section~\ref{sec:analysis}, we prove robust 
(almost) first-order convergence.  In Section~\ref{sec:numerics}, we present the results of numerical experiments which verify  our 
theoretical results.  We comment briefly on interesting avenues for future investigations in Section~\ref{sec:conclude}.
		
\paragraph*{Notation.}
We denote the domain's boundary edges, excluding the points of discontinuity, as  $\varrho=\varrho_1\cup
		\varrho_2\cup\varrho_3\cup\varrho_4$, where
\begin{equation}
\label{eq:boundaries}
\begin{aligned}
  \varrho_{1}&=\big\{(0,y)\,|\, y \in [0, d_2) \cup (d_2, 1] \big\},
 &  
  \varrho_{2}&=\big\{(x,0)\,|\, x \in [0, d_1) \cup (d_1, 1] \big\},\\
  \varrho_{3}&=\big\{(1,y)\,|\, y \in [0, d_2) \cup (d_2, 1] \big\}, 
  & 
  \varrho_{4}&=\big\{(x,1)\,|\, x \in [0, d_1) \cup (d_1, 1] \big\}.
\end{aligned}
\end{equation}
Recalling from \eqref{eqncon} that $u=q$ on the boundary, we denote by $q_i$ the restriction of $q$ onto $\varrho_i$, for $i=1,2,3,4$.
The boundaries of each quadrant, $\Omega_k$, are denoted $\varrho_{k,j}$ where $j=1,2,3,4$ indicate the edges of $\Omega_k$,
labelled anti-clockwise from the west boundary.

\section{The continuous problem}
\subsection{Stability of the differential operator}
\label{sec:a priori}
		
\begin{theorem}(Maximum principle)\label{maximumprinciple}
			Let $L_{\epsilon}$ be the differential operator given in
			(\ref{eqncon}). If $\phi(x,y)\geq 0$  on $\partial \Omega$,\,
			$L_{\epsilon}\phi(x,y)\geq 0$ for all $(x,y)\in \Omega$,
			$[\phi](d_1,y)=[\phi](x,d_2)=0$, 
			$[\frac{\partial \phi}{\partial x}](d_1,y)\leq 0$ and 
			$[\frac{\partial \phi}{\partial y}](x,d_2)\leq 0$. Then $\phi(x,y)\geq 0$
			for all $(x,y)\in \bar{\Omega}$  
		\end{theorem}
\begin{proof}
  Consider the function $\omega$ on $\bar{\Omega}$ defined as
  $\phi(x,y)=\omega(x,y)\psi(x,y)$ where the function 
  \[
    \psi(x,y)=\exp\bigg(\frac{\alpha(x-d_1)}{2\epsilon}+\frac{\beta(y-d_2)}{2\epsilon}\bigg),\quad
    (x,y)\in \bar{\Omega}
  \]
  where $\alpha>0$ and $\beta>0$ are the constants from \eqref{eq:a b bounds}.
  Let $x^\star$ and $y^\star$  be such that
  \[
  \omega(x^\star,y^\star)=\min\limits_{(x,y)\in  \bar{\Omega}}\{\omega(x,y)\}.\] 
  If $\omega(x^\star,y^\star)\geq 0$, we are done, so instead we shall assume that $\omega(x^\star,y^\star)<0$.
  At the point $(x^\star,y^\star)$ we have that   \[
    \dfrac{\partial \omega}{\partial x}{(x^\star,y^\star)}=
    \dfrac{\partial \omega}{\partial y}{(x^\star,y^\star)}=0,
    \quad \text{ and } \quad
    \dfrac{\partial^{2} \omega}{\partial x^{2}}{(x^\star,y^\star)}\geq 0, \quad
    \dfrac{\partial^{2} \omega}{\partial y^{2}}{(x^\star,y^\star)}\geq 0.
  \]
  By the assumption on the boundary values,
  either the point $(x^\star,y^\star)\in \Omega$ or $(x^\star,y^\star) \in  (d_1,y) \cup (x,d_2)$.
We consider the two cases separately.
			
\begin{description}
\item [Case(i):]   $(x^\star,y^\star)\in \Omega$.
  By our assumption that $\omega(x^\star,y^\star) < 0$, we have
  \[
    L_{\epsilon}\phi(x^\star,y^\star) =
    \bigg(-\epsilon^{2}\bigg(\frac{\partial^2\phi}{\partial
      x^2}+\frac{\partial^2\phi}{\partial
      y^2}\bigg)+a(x,y)\frac{\partial\phi}{\partial
      x}+b(x,y)\phi\bigg){(x^\star,y^\star)} < 0,
  \]
  which contradicts the hypothesis.
  
\item[Case(ii)] $(x^\star,y^\star)\in \Omega$ or $(x^\star,y^\star) \in \big((d_1,y) \cup (x,d_2)\big)$.
Here, either $(x^\star,y^\star) = (d_1,y^\star),$ or $(x^\star,y^\star) = (x^\star,d_2)$.  Let us assume $(x^\star,y^\star) = (d_1,y^\star)$. 
Then $\omega$ taking its minimum value at $(x^\star,y^\star)$ implies that 
$\frac{\partial \omega}{\partial x}({d_1}^+,y^\star) \geq 0$ and   
$\frac{\partial \omega}{\partial x}({d_1}^-,y^\star) \leq 0$. 
Then, it is evident that $[\frac{\partial \omega}{\partial x}]({d_1},y^\star) \geq 0$. As $\omega(d_1,y^\star)<0,$ it follows that
\[
\bigg[\frac{\partial \phi}{\partial x}\bigg](d_1,y^\star)=\exp\bigg(\frac{\beta(y^\star-d_2)}{2\epsilon}\bigg)\bigg(\bigg[\frac{\partial \omega}{\partial x}\bigg](d_1,y^\star)\bigg)>0,
\]
contradicting the hypothesis 
$[\frac{\partial \phi }{\partial x}](x,y)\leq 0 \quad \forall (x,y)\in (d_1,y)$. 
Similar reasoning applies when $(x^\star,y^\star) = (x^\star,d_2)$.
\end{description}
This completes the proof.
\end{proof}
		
\begin{lemma}[Stability]\label{stabilityresult}
			Let $u(x,y)$ be the solution of (\ref{eqncon}). Then 
			Theorem~\ref{maximumprinciple} holds, and yields  the stability estimate 
			\begin{equation}\label{eq:stab bound}
				||u(x,y)||_{\bar{\Omega}} \leq  \dfrac{1}{\alpha} ||f||_{\Omega}+\max \bigg\{ ||
				u ||_{\varrho}\bigg\}.
			\end{equation}
		\end{lemma}
		\begin{proof}
			We define the barrier function $\phi^{\pm}(x,y)$ as follows:
			\begin{align*}
				\phi^{\pm}(x,y)=	\begin{cases}
					M+\frac{||f||_{\Omega}}{\alpha}\bigg(1+\frac{x}{d_1}+\frac{y}{d_2}\bigg)\pm u(x,y),   &  (x,y)\in  {\Omega}_1,\\
					M+\frac{||f||_{\Omega}}{\alpha}\bigg(1+\frac{(1-x)}{(1-d_1)}+\frac{y}{d_2}\bigg)\pm u(x,y),  &  (x,y)\in  {\Omega}_2,\\
					M+\frac{||f||_{\Omega}}{\alpha}\bigg(1+\frac{x}{d_1}+\frac{(1-y)}{(1-d_2)}\bigg)\pm u(x,y),  & \, \, (x,y)\in  {\Omega}_3,\\
					M+\frac{||f||_{\Omega}}{\alpha}\bigg(1+\frac{(1-x)}{(1-d_1)}+\frac{(1-y)}{(1-d_2)}\bigg)\pm u(x,y),  &  (x,y)\in  {\Omega}_4,
				\end{cases}
			\end{align*}
			where $M = \max\{ || u ||_{\varrho}\}$.
			Then, clearly $\phi^{\pm}(x,0)\geq 0,\, \phi^{\pm}(0,y)\geq 0,\, \phi^{\pm}(x,1)\geq 0, \,\phi^{\pm}(1,y)\geq 0$. For each $(x,y)\in \Omega,$ we have 
			\[
			L_{\epsilon}\phi^{\pm}(x,y)\geq 0.
			\]
			Since, $u(x,y)\in C^{1,1}\bar{(\Omega)}\cap C^{2,2}(\Omega),$ we have
			$$ \bigg[\dfrac{\partial \phi^{\pm}}{\partial x}\bigg](d_1,y)=\frac{-||f||_\Omega}{\alpha d_1(1-d_1)}\pm \bigg[\dfrac{\partial u^{\pm}}{\partial x}\bigg](d_1,y)\leq 0,$$
			$$\bigg[\dfrac{\partial \phi^{\pm}}{\partial
				y}\bigg](x,d_2)=\frac{-||f||_\Omega}{\alpha d_2(1-d_2)}\pm
			\bigg[\dfrac{\partial u^{\pm}}{\partial y}\bigg](x,d_2)\leq 0.$$
			It follows from Theorem \ref{maximumprinciple} that
			$\phi^{\pm}(x,y)\geq 0$,\,\,  $\forall (x,y) \,\in\bar{\Omega}$, which, in
			turn, leads to \eqref{eq:stab bound}.
		\end{proof}
		
\subsection{Solution decomposition}
The solution's derivatives satisfy the following global bounds, which can be determined following arguments in
\cite{miller1998fitted}.
\begin{lemma}\label{derivativeresult1}
  Let $u$ be the solution of (\ref{eqncon}). Then, for $1\leq i+j \leq 4$,
  \begin{equation}\label{abcd}
    \bigg|\bigg|\dfrac{\partial^{i+j} u}{\partial x^{i} \partial y^{j}}\bigg|\bigg|_{\Omega}\leq C\epsilon^{-(2i+2j)}.
  \end{equation}
\end{lemma}
Our analysis requires sharper, point-wise bounds. We start by proposing a solution
decomposition. Suppose the regular component $v_k$ defined on the quadrant $\Omega_k$
can be decomposed as $v_k=v_k^0+\epsilon v_k^1$, where $v_k^0$ is the solution of the reduced problem 
\[
  a(x,y) \frac{\partial v_k^0}{\partial x} + b(x,y)v_k^0=f
  \qquad \forall (x,y)\in \Omega_k,\, v_k^0|_{\varrho_{k,1}}=0,\, k=1,2,3,4,
\]
and $v_{k}^1$ is the solution to
\begin{equation}\label{regularequation}
  L_{\epsilon}v_{k}^1=\bigg(\frac{\partial^2}{\partial x^2} + \frac{\partial^2}{\partial y^2}\bigg) v_{k}^0, \quad \forall (x,y) \in \Omega_k, \quad v_{k}^1 = 0, \,\, \forall (x,y) \in \partial \Omega_k,\,k=1,2,3,4.
		\end{equation}
		Since, $v_k^0\in C^{4,\zeta}(\bar{\Omega}_k),\, k=1,\,2,\,3,\,4$, we get
		$\bigg(\dfrac{\partial^2}{\partial x^2} + \dfrac{\partial^2}{\partial
			y^2}\bigg) v_k^0\in C^{2,\zeta}(\bar{\Omega}_k),\, k=1,\,2,\,3,\,4$. 
		The regular component $v_k,\, k=1,2,3,4$ is taken to be the solution
		of
		\begin{subequations}
			\begin{align}
				\label{regresult}
				L_{\epsilon} v_k = f, \,\, \forall (x,y) \in \Omega_k,\\
				v_k(x,y)=q_1(y),\quad \forall(x,y)\in \varrho_1,\quad
				v_k(x,y)=q_3(x),\quad  \forall (x,y)\in \varrho_3,\\
				v_k(x,y)=q_2(y),\quad \forall(x,y)\in \varrho_2,\quad
				v_k(x,y)=q_4(x),\quad  \forall (x,y)\in \varrho_4,\\
                          [v_k](d_1,y)=0,\quad[(v_k)_x](d_1,y)=0,\\
				\label{regresult1}
                          [v_k](x,d_2)=0,\quad[(v_k)_y](x,d_2)=0.
			\end{align}
		\end{subequations}
Applying Lemma~\ref{stabilityresult} and Lemma~\ref{derivativeresult1}
to (\ref{regularequation}), we deduce that $v_k\, \in
\,C^{4,\zeta}({\Omega_k})$ and   
\begin{equation}\label{regularderieqn}
  \bigg|\bigg|\dfrac{\partial^{i+j} v_k}{\partial x^i \partial y^j}\bigg|\bigg|\leq C(1+\epsilon^{4-(2i+2j)}), \quad 0\leq i+j \leq 4,\,\, k=1,\,2,\,3,\,4.
\end{equation}

Corresponding to the eastern boundary of  $\Omega_2$
(see Figure~\ref{fig:Omega}), the layer function $w_{R_1}$ is defined as the solution to
\begin{subequations}
  \begin{align}
    \label{singularequation11} L_{\epsilon}w_{R_1}=0,\quad \forall (x,y)\in \Omega_2, \\ 
    w_{R_1}(x,y)= (u - v_2)(x,y), \quad \forall (x,y) \in \varrho_{2,3},\\
    w_{R_1}(x,y)=0, \quad \forall (x,y) \in \varrho_{2,2},\\
    [w_{R_1}](d_1,y)=0, \quad
    [(w_{R_1})_x](d_1,y)=0, \\ 
    [w_{R_1}](x,d_2)=0,\quad
    [(w_{R_1})_y](x,d_2)=0.
  \end{align}
\end{subequations}
		
		Set $\varsigma=(1-x)/\epsilon^2$ and
		${w}_{R_1}(x,y)=\bar{w}_{R_1}(\varsigma,y).$
		We  write $\bar{L}_{\epsilon}$ for the operator $L_{\epsilon}$ defined
		in term of the variables $y$ and $\varsigma$. Then, 
		\[
		\bar{L}_{\epsilon}\bar{w}_{R_1}=-\epsilon^{-2}\frac{\partial^2\bar{w}_{R_1}}{\partial
			\varsigma^2}-\epsilon^2 \frac{\partial^2\bar{w}_{R_1}}{\partial
			y^2}-\epsilon^{-2}\bar{a}(\varsigma,y)\frac{\partial
			\bar{w}_{R_1}}{\partial
			\varsigma}+\bar{b}(\varsigma,y)\bar{w}_{R_1},
		\]
		where, $\bar{a}(\varsigma,y)=a(1-\epsilon^2\varsigma,y)$ and
		$\bar{b}(\varsigma,y)=b(1-\epsilon^2\varsigma,y).$ 
		Expanding $\bar{a}(\varsigma,y)$ and $\bar{b}(\varsigma,y)$ in powers
		of $\epsilon^2 \varsigma$, 
		we have  
		\begin{align*}
			\bar{a}(\varsigma,y) &=\sum_{l=0}^{\infty}\frac{(-\epsilon^2\varsigma)^l}{l!}\frac{\partial^l
				a}{\partial x^l}(1,y),\\
			\bar{b}(\varsigma,y) &=\sum_{l=0}^{\infty}\frac{(-\epsilon^2\varsigma)^l}{l!}\frac{\partial^l
				b}{\partial x^l}(1,y).
		\end{align*}
		Define, $\bar{w}_{R_1}=\bar{w}_{0,R_1}+\epsilon^2\bar{w}_{1,R_1}.$
		With the above change of variables, 
		equating the coefficients of $\epsilon^{0}$ and $\epsilon^{-2}$
		in $w_{R_1}$  that solves (\ref{singularequation11}),
		we deduce that $\bar{w}_{0,R_1}$ and $\bar{w}_{1,R_1}$, satisfy the equations
		\begin{subequations}
			\begin{align}\label{singulartwo}
				\frac{\partial^2\bar{w}_{0,R_1}}{\partial\varsigma^2}+a(1,y)\frac{\partial \bar{w}_{0,R_1}}{\partial \varsigma}=0,\\
				\frac{\partial^2\bar{w}_{1,R_1}}{\partial\varsigma^2}+a(1,y)\frac{\partial \bar{w}_{1,R_1}}{\partial \varsigma}=b(1,y)\bar{w}_{0,R_1}+\varsigma \frac{\partial a(1,y)}{\partial x}\frac{\partial \bar{w}_{0,R_1}}{\partial \varsigma}, 
			\end{align} 
			with boundary conditions
			\begin{align}\label{boundarysingular}
				\bar{w}_{0,R_1}(0,y)=q_3(y)-v_2(1,y), \, \bar{w}_{1,R_1}(0,y)=-v_1(1,y),\,  \bar{w}_{s,R_1}(\varsigma,y)=0\quad as \quad \varsigma\to \infty.
			\end{align}
		\end{subequations}
		Solving (\ref{singulartwo})--(\ref{boundarysingular}),
		we obtain $\bar{w}_{0,R_1}(\varsigma,
		y)=(q_3(y)-v_2(1,y))\exp(-{a(1,y)}\varsigma)$
		and 
		$\bar{w}_{1,R_1}(\varsigma,
		y)=\big\{-v_1+\varsigma\big[\frac{c-a}{a}(q_3(y)-v_2)\big]+\varsigma^2\frac{\partial
			a}{\partial x}(v_2-q_3)\big\}(1,y)\exp(-{a(1,y)}\varsigma)$. Using
		$\min\limits_{(x,y)\in\bar{\Omega}}a(x,y)>\alpha$,  it follows that
		\begin{equation*}
			\bigg|\dfrac{\partial^{i+j} {w}_{s,R_1}}{\partial x^i \partial y^j}\bigg|\leq  C \epsilon^{-2i}\exp(-\alpha (1-x)/\epsilon^2),\quad 0\leq i+j \leq 4, \quad s = 0,1. 
		\end{equation*}
This implies the bounds on the boundary layer components are
\begin{equation}\label{singularresult}
  \bigg|\dfrac{\partial^{i+j} {w}_{R_1}}{\partial x^i \partial y^j}\bigg|\leq  C \epsilon^{-2i}\exp(-\alpha (1-x)/\epsilon^2),\quad 0\leq i+j \leq 4.
\end{equation}
In a similar fashion, one may construct a another boundary layer function, $w_{R_2}$, associated with right-hand
boundary of $\Omega_4$ (see Figure~\ref{fig:Omega}).
		
Corresponding to the south  edge of  quadrant $\Omega_2$,   we
define ${w}_{B_2}$, associated with the characteristic layer, as the
function satisfying
\begin{subequations}
  \begin{align}
    \label{singularequation111} L_{\epsilon}{w}_{B_2}=0,\quad \forall (x,y)\in \Omega_2, \\ 
    {w}_{B_2}(x,y)= (u - v_2)(x,y), \quad \forall (x,y) \in \varrho_{2,2},\\
    {w}_{B_2}(x,y)=0, \quad \forall (x,y) \in \varrho_{2,3},\\
     [{w}_{B_2}](d_1,y)=0, \quad [({w}_{B_2})_x](d_1,y)=0, \\
     [{w}_{B_2}](x,d_2)=0,\quad  [({w}_{B_2})_y](x,d_2)=0.
  \end{align}
\end{subequations}

		As before, we employ of change of variables, and set
		$\eta=-y/\epsilon$, ${w}_{B_2}(x,y)=\bar{w}_{B_2}(x,\eta)$
		corresponding to ${w}_{B_2}$.
		we write $\tilde{L}_\epsilon$ for the operator $L_\epsilon$
		defined in terms of the variables $x$ and $\eta.$ Then,  
		\[
		\tilde{L}_{\epsilon}\tilde{{w}}_{B_2}=-\epsilon^{-2}\frac{\partial^2\tilde{w}_{B_2}}{\partial
			x^2}- \frac{\partial^2\tilde{w}_{B_2}}{\partial
			\eta^2}+\tilde{a}(x,\eta)\frac{\partial \tilde{w}_{B_2}}{\partial
			x}+\tilde{b}(x,\eta)\tilde{w}_{B_2},
		\]
		where, $\tilde{a}(x,\eta)=a(x,\epsilon\eta)$ and
		$\tilde{b}(x,\eta)=b(x,\epsilon\eta).$
		Expanding $\tilde{a}(x,\eta)$ and $\tilde{b}(x,\eta)$ in powers of
		$\epsilon \eta$,  we have 
		\begin{align}
			\tilde{a}(x,\eta) &=\sum_{l=0}^{\infty}\frac{(-\epsilon\eta)^l}{l!}\frac{\partial^l a}{\partial y^l}(x,0),\\
			\tilde{b}(x,\eta) &=\sum_{l=0}^{\infty}\frac{(-\epsilon\eta)^l}{l!}\frac{\partial^l
				b}{\partial y^l}(x,0).
		\end{align}
		Define $\tilde{w}_{B_2}:=\tilde{w}_{0,B_2}+\epsilon\tilde{w}_{1,B_2}$.
		With the  above change of variables, 
		and equating the coefficients of $\epsilon^0$ and $\epsilon^{1}$, we deduce
		that
		$\tilde{w}_{0,B_2}$ and $\tilde{w}_{1,B_2}$ satisfy
		\begin{subequations}
			\begin{align}\label{singularthree}
				\frac{\partial^2 \tilde{w}_{0,B_2} }{\partial
					\eta^2}-b(x,0)\tilde{w}_{0,B_2}=0,\\
				\frac{\partial^2 \tilde{w}_{1,B_2}}{\partial
					\eta^2}-\eta\frac{\partial a(x,0)}{\partial y}\frac{\partial
					\tilde{w}_{0,B_2}}{\partial x}-\eta\frac{\partial b(x,0)}{\partial
					y}\tilde{w}_{0,B_2}+b(x,0)\tilde{w}_{1,B_2}=0, 
			\end{align}
			with boundary conditions
			\begin{equation}
				\label{boundarysingulartwo}
				\tilde{w}_{0,B_2}(x,0)=q_2(x)-v_2(x,0),\,\,
				\tilde{w}_{1,B_2}=-v_2(x,0),\, \tilde{w}_{s,B_2}(x,\eta)=0\quad
				as\quad \eta\to \infty.
			\end{equation}
		\end{subequations}
		
		Solving (\ref{singularthree})--(\ref{boundarysingulartwo}), we get
		$\tilde{w}_{0,B_2}(x,\eta)=(q_2(x)-v_2(x,0))\exp(-b(x,0)\eta)$ and
		$\tilde{w}_{1,B_2}=\{-v_2+\eta[\frac{\partial}{\partial
			y}(a)(v_2-q_2)+\frac{b-a}{a}(v_2-q_2)]\}(x,0)\exp(-b(x,0)\eta)$. Then,
		$\bigg|\dfrac{\partial^{i+j} \tilde{w}_{s,B_2}}{\partial x^i \partial
			y^j}\bigg|(x,\eta)\leq C\exp(-\beta\eta)$ and hence
		$\bigg|\dfrac{\partial^{i+j} {w}_{s,B_2}}{\partial x^i \partial
			y^j}\bigg|(x,y)\leq C\epsilon^{-j}\exp(-\beta y/\epsilon)$ for $0\leq
		i+j\leq 4 $ and $s=0,1$. 
	The same reasoning may be applied to the corresponding  bottom, top boundary and interior
		layer terms
		$w_{B_1}$, $w_{T_1}, w_{T_2},\,w_{D_1},\,w_{D_2}, \, w_{D_3},\,w_{D_4},\,w_{D_5}$, and $w_{D_6}$, which correspond the
                  regions pictured in
                  Figure~\ref{fig:Omega}.

		Next we introduce the  corner layer function $w_{B_2R_1}$,
		associated with the corner
		of the domain at (0,0),  and defined as the solution to
		\begin{subequations}
			\begin{align}
				\label{cornerlayer}
			L_{\epsilon}{w}_{B_2R_1}=0,\quad \forall (x,y)\in \Omega_2, \\ 
			{w}_{B_2R_1}(x,y)=  -w_{R_1}(x,y), \quad \forall (x,y) \in \varrho_{2,3},\\
			{w}_{B_2R_1}(x,y)=  -w_{B_2}(x,y), \quad \forall (x,y) \in \varrho_{2,2},\\
			{[{w}_{B_2R_1}](d_1,y)=0, \quad [({w}_{B_2R_1})_x](d_1,y)=0,} \\
			{[{w}_{B_2R_1}](x,d_2)=0, \quad [({w}_{B_2R_1})_y](x,d_2)=0}.
			\end{align}
		\end{subequations}
		We note that $L_{\epsilon} w_{B_2} = L_{\epsilon} w_{R_1} = 0$ and
		$w_{B_2},w_{R_1}  \in C^{4}(\Omega_2).$ Therefore, the compatibility
		conditions up to second-order are satisfied at
		the four corners of the domain which indicates that $w_{B_2R_1} \in
		C^{4}(\Omega_2)$; see \cite{han1990differentiability}.
		Now, set $\varsigma=(1-x)/\epsilon^2,\, \eta=-y/\epsilon$ and the
		corner layer function
		$w_{B_2R_1}(x,y)=\bar{w}_{B_2R_1}(\varsigma,\eta)$ in
		(\ref{cornerlayer}) corresponding to $w_{B_2R_1}.$ Furthermore, we
		write $\breve{L}_\epsilon$ for the operator $L_\epsilon$ defined in
		terms of the variables $\varsigma$ and $\eta$. Then,   
		\[
		\breve{L}\breve{w}_{B_2R_1}=-\epsilon^{-2}\frac{\partial^2\breve{w}_{B_2R_1}}{\partial
			\varsigma^2}-\frac{\partial^2\breve{w}_{B_2R_1}}{\partial
			\eta^2}-\breve{a}(\varsigma,\eta)\epsilon^{-2}\frac{\partial
			\breve{w}_{B_2R_1}}{\partial\varsigma}+\breve{b}(\varsigma,\eta)\breve{w}_{B_2R_1},
		\]
		where,
$\breve{a}(\varsigma,\eta)=a(1-\epsilon^2\varsigma,-\epsilon\eta)$ and
$b(\varsigma,\eta)=b(1-\epsilon^2\varsigma,-\epsilon\eta)$. Expanding
$\breve{a}(\varsigma,\eta)$ and $\breve{b}(\varsigma,\eta)$ in powers
of $\epsilon^2\varsigma$ and $\epsilon\eta$, we
have 
\[
  \breve{a}(\varsigma,\eta)=\sum_{l=0}^{\infty}\sum_{m=0}^{\infty}\frac{(-\epsilon\eta)^l(-\epsilon^2\varsigma)^m}{l!m!}\frac{\partial^{l+m}
    a}{\partial x^l\partial y^m}(1,0),
\]
\[
  \breve{b}(\varsigma,\eta)=\sum_{l=0}^{\infty}\sum_{m=0}^{\infty}\frac{(-\epsilon\eta)^l(-\epsilon^2\varsigma)^m}{l!m!}\frac{\partial^{l+m}
    b}{\partial x^l\partial y^m}(1,0). 
\]
Let
$\breve{w}_{B_2R_1}=\breve{w}_{0,B_2R_1}+\epsilon\breve{w}_{1,B_2R_1},$
where $\breve{w}_{s,B_2R_1},$ for $s=0,1$ are the solutions to the
problems
\[
  -\frac{\partial^2\breve{w}_{0,B_2R_1}}{\partial \eta^2}+\varsigma
  \frac{\partial a(1,0)}{\partial x}\frac{\partial
    \breve{w}_{0,B_2R_1}}{\partial\varsigma}+b(1,0)\breve{w}_{0,B_2R_1}=0,
\]
\[
  -\frac{\partial^2\breve{w}_{1,B_2R_1}}{\partial\varsigma^2}-\eta\frac{\partial
    a(1,0)}{\partial y}\frac{\partial\breve{w}_{0,B_2R_1}}{\partial
    \varsigma}-a(1,0)\frac{\partial \breve{w}_{1,B_2R_1}}{\partial
    \varsigma}=0,
\]
with the boundary conditions 
\[
  \breve{w}_{s,B_2R_1}(\varsigma,0)=-v_s(\varsigma,0),\,\,
  \breve{w}_{s,B_2R_1}(1,\eta)=-v_s(1,\eta),\,\,
  \breve{w}_{s,B_2R_1}(\varsigma,\eta)=0,
\]
as $\varsigma,\eta\to\infty$, for $s=0,1$.
Using the arguments in~\cite{linss2001asymptotic}, it follows that
\[
  \bigg|\frac{\partial^{i+j}{w}_{B_2R_1}}{\partial x^i\partial
    y^j}\bigg|\leq C
  \epsilon^{-2i-j}\exp(-\alpha(1-x)/\epsilon^2)\exp(-\beta
  y/\epsilon),\quad 0\leq i+j\leq 4.
\]  

The general methodology outlined here can be adapted to give the bounds for the first
derivatives of the remaining seven corner layer functions. We can now
conclude with the following theorem.
		
\begin{theorem}\label{decompostiontheorem}
  The solution $u$, to \eqref{eq:full problem}
  may be written as a sum
  \begin{multline*}
    u = \sum_{k=1}^{4}v_{k}
    +\sum_{i=1}^{2}({w}_{B_i}+{w}_{T_i}+{w}_{R_i})
    + \sum_{j=1}^{6}{w}_{D_j} \\
    + {w}_{B_2R_1} + {w}_{D_2R_1} +
    {w}_{D_4R_2} + {w}_{T_2R_2}+{w}_{D_5B_1}+{w}_{D_5D_1}+{w}_{D_6D_3}+{w}_{D_6T_1},
  \end{multline*}
  where, 
  \[
    L_{\epsilon}v_k = f, \quad
    L_{\epsilon}{w}_{B_i}=0,\,L_{\epsilon}{w}_{T_i}=0,\,
    L_{\epsilon}{w}_{R_i} = 0,\, L_{\epsilon} {w}_{D_j}=0, \quad
    k=1,2,3,4,\,\,i=1,2,\,\, j=1,2,..,6;
  \]
  \[
    L_{\epsilon}{w}_{B_2R_1}=0,\, L_{\epsilon}{w}_{D_2R_1}=0,\,
    L_{\epsilon}{w}_{D_4R_2}=0,\, L_{\epsilon}{w}_{T_2R_2}=0,
  \]
  \[L_{\epsilon}{w}_{D_5B_1}=0,\, L_{\epsilon}{w}_{D_5D_1}=0,\,
    L_{\epsilon}{w}_{D_6D_3}=0,\, L_{\epsilon}{w}_{D_6T_1}=0.\]
  Boundary conditions for the various components can be defined so that
  \begin{equation*}
    \bigg\|\dfrac{\partial^{i+j} v_k}{\partial x^i \partial y^j}\bigg\|\leq C(1+\epsilon^{4-(2i+2j)}),\quad  0\leq i+j\leq 4,\, k=1,2,3,4,
  \end{equation*}
  \begin{align*}
    |{w}_{B_1}(x,y)| &\leq C e^{\frac{-\beta}{\epsilon} y}, &
                                                              |{w}^{}_{D_5}(x,y)| &\leq C \epsilon e^{\frac{-\alpha}{\epsilon^2} (d_1-x)},\\
    |{w}^{}_{B_2}(x,y)| &\leq C e^{\frac{-\beta}{\epsilon} y}, &
                                                              |{w}^{}_{D_6}(x,y)| &\leq C \epsilon e^{\frac{-\alpha}{\epsilon^2} (d_1-x)},\\
    |{w}^{}_{R_1}(x,y)| &\leq C e^{\frac{-\alpha}{\epsilon^2}(1-x)}, &
                                                                    |{w}^{}_{B_2R_1}(x,y)| &\leq Ce^{\frac{-\alpha}{\epsilon^2}(1-x)}e^{\frac{-\beta}{\epsilon}y},\\
    |{w}^{}_{R_2}(x,y)| &\leq C e^{\frac{-\alpha}{\epsilon^2}(1-x)}, &
                                                                    |{w}^{}_{D_2R_1}(x,y)| &\leq Ce^{\frac{-\alpha}{\epsilon^2}(1-x)}e^{\frac{-\beta}{\epsilon}(d-y)},\\
    |{w}^{}_{T_1}(x,y)| &\leq C e^{\frac{-\beta}{\epsilon} (1-y)}, &
                                                                  |{w}^{}_{D_4R_2}(x,y)| &\leq Ce^{\frac{-\alpha}{\epsilon^2}(1-x)}e^{\frac{-\beta}{\epsilon}(y-d_2)},\\
    |{w}^{}_{T_2}(x,y)| &\leq C e^{\frac{-\beta}{\epsilon} (1-y)}; &
                                                                  |{w}^{}_{T_2R_2}(x,y)| &\leq Ce^{\frac{-\alpha}{\epsilon^2}(1-x)}e^{\frac{-\beta}{\epsilon}(1-y)},\\
    |{w}^{}_{D_1}(x,y)| &\leq C e^{\frac{-\beta}{\epsilon} (d-y)}, &
                                                                  |{w}^{}_{D_5B_1}(x,y)| &\leq Ce^{\frac{-\alpha}{\epsilon^2}(d_1-x)}e^{\frac{-\beta}{\epsilon}y},\\
    |{w}^{}_{D_2}(x,y)| &\leq C e^{\frac{-\beta}{\epsilon} (d_2-y)}, & |{w}^{}_{D_5D_1}(x,y)| &\leq Ce^{\frac{-\alpha}{\epsilon^2}(d_1-x)}e^{\frac{-\beta}{\epsilon}(d_2-y)},\\
    |{w}^{}_{D_3}(x,y)| &\leq C e^{\frac{-\beta}{\epsilon} (y-d_2)}, & |{w}^{}_{D_6D_3}(x,y)| &\leq Ce^{\frac{-\alpha}{\epsilon^2}(d_1-x)}e^{\frac{-\beta}{\epsilon}(y-d_2)},\\
    |{w}^{}_{D_4}(x,y)| &\leq C e^{\frac{-\beta}{\epsilon} (y-d_2)}, & |{w}^{}_{D_6T_1}(x,y)| &\leq Ce^{\frac{-\alpha}{\epsilon^2}(d_1-x)}e^{\frac{-\beta}{\epsilon}(1-y)},\\
  \end{align*}
  \begin{equation*}
    \bigg\|\dfrac{\partial^l  {w}^{}_{B_j}}{\partial x^l}\bigg\|   \leq C\epsilon^{4-2l},\,\, \bigg\|\dfrac{\partial^k  {w}^{}_{B_j}}{\partial y^k}\bigg\|   \leq C\epsilon^{-k}\quad 	where \,\,\, j = 1,2 \,\, \, and \,\,\, 2\leq l,k\leq 4,
  \end{equation*}
  \begin{equation*}
    \bigg\|\dfrac{\partial^l  {w}^{}_{T_j}}{\partial x^l}\bigg\|   \leq C\epsilon^{4-2l},\,\, \bigg\|\dfrac{\partial^k  {w}^{}_{T_j}}{\partial y^k}\bigg\|   \leq C\epsilon^{-k}\quad 	where \,\,\, j = 1,2 \,\, \, and \,\,\, 2\leq l,k\leq 4,
  \end{equation*}
  \begin{equation*}
				\bigg\|\dfrac{\partial^l  {w}^{}_{D_j}}{\partial x^l}\bigg\|   \leq C\epsilon^{4-2l},\,\, \bigg\|\dfrac{\partial^k  {w}^{}_{D_j}}{\partial y^k}\bigg\|   \leq C\epsilon^{-k}\quad 	where \,\,\, j = 1,2,3,4 \,\, \, and \,\,\, 2\leq l,k\leq 4,
			\end{equation*}
		    \begin{equation*}
			\bigg\|\dfrac{\partial^l  {w}^{}_{D_j}}{\partial x^l}\bigg\|   \leq C\epsilon^{2-2l},\,\, \bigg\|\dfrac{\partial^k  {w}^{}_{D_j}}{\partial y^k}\bigg\|   \leq C\epsilon^{2-2k}\quad 	where \,\,\, j = 5,6 \,\, \, and \,\,\, 2\leq l,k\leq 4,
		    \end{equation*}
			\begin{equation*}
				\bigg\|\dfrac{\partial^l  {w}^{}_{R_j}}{\partial x^l}\bigg\|   \leq C\epsilon^{-2l},\,\, \bigg\|\dfrac{\partial^k  {w}^{}_{R_j}}{\partial y^k}\bigg\|   \leq C\epsilon^{2-2k}\quad
				where \,\,\, j = 1,2 \,\, \, and \,\,\, 1\leq l\leq 4,\,2\leq k\leq 4,
			\end{equation*}
			\begin{equation*}
				\bigg\|\dfrac{\partial^l  {w}^{}_{B_2R_1}}{\partial x^l}\bigg\|   \leq C\epsilon^{-2l},\,\, \bigg\|\dfrac{\partial^k  {w}^{}_{B_2R_1}}{\partial y^k}\bigg\|   \leq C\epsilon^{2-2k}\quad
				where \,\,\, 1\leq l,k\leq 4.
			\end{equation*}
Analogous bounds exists  for the remaining corner layer function ${w}_{D_2R_1},\,{w}_{D_4R_2},\,{w}_{T_2R_2},\\\,{w}_{D_5B_1},\,{w}_{D_5D_1},\,{w}_{D_6D_3},\,{w}_{D_6T_1}$ 
\end{theorem}

The solution $u$ of (\ref{eqncon}) can be written as,
\begin{align*}
  u(x,y)= 
  \begin{cases}
    (v_1 + {w}_{B_1} + {w}_{D_1}+{w}_{D_5}+{w}_{D_5B_1}+{w}_{D_5D_1} )(x,y),\,\,\, \quad \qquad \forall (x,y) \in \Omega_1\\
    (v_2 + {w}_{B_2} + {w}_{R_1} + {w}_{D_2} + {w}_{B_2R_1} + {w}_{D_2R_1})(x,y),\,\,\, \quad \qquad \forall (x,y) \in \Omega_2\\
    (v_3 + {w}_{D_3} +{w}_{D_6}+ {w}_{T_1}+{w}_{D_6D_3}+{w}_{D_6T_1} )(x,y), \,\,\, \quad \qquad \forall (x,y) \in \Omega_3\\
    (v_4 + {w}_{T_2} + {w}_{D_4} +{w}_{R_2}+  {w}_{D_4R_2} + {w}_{T_2R_2})(x,y), \,\qquad\quad\,\, \forall (x,y) \in \Omega_4\\
    (v + w + w') ({d_1}^+,y) = (v + w + w')({d_1}^-,y), \\
    (v + w + w') (x,{d_2}^+) = (v + w + w')(x,{d_2}^-),
  \end{cases}
\end{align*}
where,
\[
  v = \sum_{i=1}^{4}v_{i}, \qquad
  w =
  \sum_{i=1}^{2}({w}_{B_i}+{w}_{T_i}+{w}_{R_i})+\sum_{j=1}^{6}{w}_{D_j},
\]
and
\[
w'
={w}_{B_2R_1}+{w}_{D_2R_1}+{w}_{D_4R_2}+{w}_{T_2R_2}+{w}_{D_5B_1}+{w}_{D_5D_1}+{w}_{D_6D_3}+{w}_{D_6T_1}.
\]
		
\section{Numerical method}\label{sec:FDM}
\subsection{A fitted mesh}	\label{sec:fitted mesh}	
We define the mesh transition points
\begin{equation}\label{eq:transition points}
  \sigma_x = \min\bigg\{\frac{d_1}{2},\frac{2 \epsilon^2}{\alpha}\ln
  N\bigg\},
  \quad \text{ and }
  \sigma_y = \min\bigg\{\frac{d_2}{4},\frac{2
    \epsilon}{\beta}\ln N\bigg\}.
\end{equation}
We subdivide the unit interval in $x$ into
four subdomains
\begin{subequations}\label{eq:sub intervals}
  \begin{equation}
    \label{eq:Omega x}
    [0,1]=[0,d_1-\sigma_x] \cup [d_1-\sigma_x,d_1]\cup [d_1,1-\sigma_x] \cup
  [1-\sigma_x,1].
\end{equation}
Let $\overline{\Omega}_x^N$ denote the one-dimensional piecewise 
uniform mesh obtained by placing a uniform mesh with
$N/4$ mesh intervals on each of these four subdomains.
For convenience, 
we denote $\Omega_x^N$ the mesh points in  $\overline{\Omega}_x^N$ that exclude
the boundary and discontinuity points $\{0,  d_1, 1\}$. 
Similarly, we subdivide $[0,1]$ into  six  subintervals
\begin{equation}    \label{eq:Omega y}
  [0,1]=[0,\sigma_y] \cup [\sigma_y,d_2-\sigma_y] \cup [d_2-\sigma_y,d_2]
  \cup [d_2,d_2+\sigma_y] \cup [d_2+\sigma_y,1-\sigma_y] \cup
  [1-\sigma_y,1],
\end{equation}
Let $\overline{\Omega}_y^N$ then denote the one-dimensional piecewise 
uniform mesh obtained by placing a uniform mesh on each of the
subintervals in~\eqref{eq:Omega y}, with 
$N/4$ mesh intervals on both $[\sigma_y,d_2-\sigma_y]$ and
$[d+\sigma_y,1-\sigma_y]$, and  $N/8$ on the other four.
Again, for convenience,  we denote $\Omega_y^N \equiv
\overline{\Omega}_y^N \setminus \{0,  d, 1\}$. 
Finally, we set
$\bar{\Omega}^{N,N} := \bar{\Omega}_x^N \times \bar{\Omega}_y^N$.
That is, $\bar{\Omega}^{N,N}$ is the
piecewise two-dimensional tensor product mesh with mesh points
$(x_i, y_j)$ with $x_i \in \bar \Omega_x^N$, $y_j \in \bar \Omega_y^N$. 
Similarly, we set ${\Omega}^{N,N} := {\Omega}_x^N \times {\Omega}_y^N$.
Later, we shall let 
$\Omega_i^{N,N}$ be the mesh restricted to the subdomain
$\Omega_{i}$ as defined in \eqref{eq:subdomains}.
\end{subequations}

\begin{remark}\label{rem:d big enough}
  For \eqref{eq:transition points} to make sense when used in
  \eqref{eq:sub intervals}, it must be that  $\epsilon$ sufficiently small
  so that $\{d_1,\,d_2\} > 8(\epsilon/\beta)\ln N$. This
  is not a significant restriction for this study since we are focused
  on the case where  $\epsilon$ maybe arbitrarily small.
\end{remark}

We shall employ the notation that $h_i:= x_i-x_{i-1}$, and
$k_j:=y_j-y_{j-1}$. But since there are few actual distinct mesh
widths, it is useful to introduce notation for them:
\begin{gather}\label{stepsizes}
  \begin{aligned}
  H_1&= \dfrac{4(d_1-\sigma_x)}{N},\qquad  &
  h_1&= \dfrac{4\sigma_x}{N}=h_2, \qquad&
  H_2&= \dfrac{4(1-\sigma_x-d_1)}{N}, \\
  K_1&= \dfrac{4(d_2-2\sigma_y)}{N}, &
  k_1&=   \dfrac{8\sigma_y}{N}=k_2, &
  K_2&= \dfrac{4(1-2\sigma_y -d_2)}{N}.
\end{aligned}
\end{gather}

\subsection{The finite difference method}
Let $U \equiv U(x_i,y_j)$ denotes a mesh function on the mesh of
Section~\ref{sec:fitted mesh} (usually,
a finite difference approximation of the exact solution $u(x,y)$).
We define the usual first-order and second-order
discrete differential operators
\[
  D_x^-U(x_i,y_j)=\frac{U(x_i,y_j)-U(x_{i-1},y_j)}{h_i}, \quad
  D_y^-U(x_i,y_j)=\frac{U(x_i,y_j)-U(x_{i},y_{j-1})}{k_j},
\]
\[D_x^+U(x_i,y_j)=\frac{U(x_{i+1},y_j)-U(x_{i},y_j)}{h_{i+1}},\quad
  D_y^+U(x_i,y_j)=\frac{U(x_i,y_{j+1})-U(x_{i},y_j)}{k_{j+1}},\]
and
\[
  \delta_{xx}^2U(x_i,y_j)=\frac{1}{\overline{h}_i}(D_x^+U(x_i,y_j)-D_x^-U(x_i,y_j),
  \quad 
  \delta_{yy}^2U(x_i,y_j)=\frac{1}{\overline{k}_j}(D_y^+U(x_i,y_j)-D_y^-U(x_i,y_j)).
\] 
              
On the mesh from Section~\ref{sec:fitted mesh},	
we define a standard upwinding finite difference
operator
\begin{equation}\label{upeqn}
	L_{u,\epsilon}^{N,N}U(x_i,y_j) \equiv
	-\epsilon^2(\delta_{xx}^2+\delta_{yy}^2)U(x_i,y_j)+a_{i,j}D_x^-U(x_i,y_j)+b_{i,j}U(x_i,y_j),
\end{equation}
where, for  convenience, we use the shorthand 
$a_{i,j}=a(x_i, y_j)$ and 
$b_{i,j}=b(x_i, y_j)$.

We also define two three-point one-sided difference operators
\[
D_x^{--}U(x_i,y_j)=
\frac{2h_{i} + h_{i-1}}{h_{i}(h_{i} + h_{i-1})}U(x_i, y_j)
- \frac{h_{i} + h_{i-1}}{h_{i}h_{i-1}}U(x_{i-1}, y_j)
+ \frac{h_{i}}{h_{i-1}(h_{i} + h_{i-1})}U(x_{i-2}, y_j)
\]
and
\begin{multline*}
D_x^{++}U(x_i,y_j)=
-\frac{2h_{i+1} + h_{i+2}}{h_{i+1}(h_{i+1} + h_{i+2})}U(x_{i}, y_j) +
\frac{h_{i+1} + h_{i+2}}{h_{i+1}h_{i+1}}U(x_{i+1}, y_j) \\-
\frac{h_{i+1}}{h_{i+2}(h_{i+1} + h_{i+2})}U(x_{i+2}, y_j).
\end{multline*}
These are employed along the line $(x_{N/2},y_j)=(d_1,y_j)$
where they simplify to
\begin{multline*}
  L_{t,\epsilon}^{N,N}{U}(x_{N/2},y_j) \equiv
  D_x^{--}U(x_{N/2},y_j) - D_x^{++}U(x_{N/2},y_j) \\=
  \dfrac{-U(x_{N/2+2},y_j)+4U(x_{N/2+1},y_j)-3U(x_{N/2},y_j)}{2H_2}\\-\dfrac{U(x_{N/2-2},y_j)-4U(x_{N/2-1},y_j)+3U(x_{N/2},y_j)}{2h_1}, 
\end{multline*}
Along the line $(x_{i},y_{N/2})=(x_i,d_2)$ we shall use the midpoint upwind scheme
\begin{equation*}
	L_{m,\epsilon}^{N,N}{U}(x_i,y_j) \equiv
	-\epsilon^2({\delta}_{xx}^2+{\delta}_{yy}^2){U}(x_i,y_j)+\hat{a}_{i,j}D_x^{-}{U}(x_i,y_j)+\hat{b}_{i,j}{U}(x_i,y_j),
\end{equation*}
where  \[
\hat{a}_{i,N/2}=\dfrac{{a}_{i,N/2-1}+ {a}_{i,N/2+1}}{2}, \quad
\hat{b}_{i,N/2}=\dfrac{{b}_{i,N/2-1}+ {b}_{i,N/2+1}}{2}.
\]
Similarly, we denote
\[
\hat{f}(x_i,y_{N/2})=\dfrac{{f}(x_i,y_{N/2-1})+ {f}(x_i,y_{N/2+1})}{2}.
\]
Then the finite difference method,  on the fitted mesh,  is defined as 
\begin{equation}\label{diseqn}
  L_{\epsilon*}^{N,N}U(x_i,y_j)=
  \begin{cases}
    L_{u,\epsilon}^{N,N}U(x_i,y_j) = f(x_i, y_j), & \forall
    (x_i,y_j)\in {\Omega}^{N,N},\\     
    L_{t,\epsilon}^{N,N}{U}(x^{}_{N/2},y_j)=0, & 1\leq j\leq N, \\
    L_{m,\epsilon}^{N,N}{U}(x_i,y^{}_{N/2}) = \hat{f}(x_i,y^{}_{N/2}), &  1\leq i\leq N,\\
    U(x_i,y_j)=q_k(y_j), & (x_i,y_j)\in \varrho_k^{N,N}, ~ k=1,3,\\
    U(x_i,y_j)=q_k(x_i), & (x_i,y_j)\in \varrho_k^{N,N}, ~ k=2,4.
  \end{cases}
\end{equation}
The matrix associated with (\ref{diseqn})  is not an M-matrix. We
transform the equations so that the new equations do have a
monotonicity property. From  (\ref{upeqn}), we
get
\begin{multline*}
  U(x_{N/2-2},y_j)=\frac{h_1^2}{\epsilon^2+h_1a_{N/2-1,j}}\bigg[\frac{-\epsilon^2}{h_1^2}U(x_{N/2},y_j) \\ + \bigg(\frac{2\epsilon^2}{h_1^2}+\frac{a_{N/2-1,j}}{h_1}+\frac{b_{N/2-1,j}}{2}\bigg)U(x_{N/2-1},y_j)  -\frac{f(x_{N/2-1},y_j)}{2}\bigg],
  \end{multline*}
  \begin{multline*}
  U(x_{N/2+2},y_j)=\frac{H_2^2}{\epsilon^2}\bigg[-\bigg(\frac{\epsilon^2}{H_2^2} \\ +\frac{a_{N/2+1,j}}{H_2}\bigg)U(x_{N/2},y_j)+\bigg(\frac{2\epsilon^2}{H_2^2}+\frac{a_{N/2+1,j}}{H_2}+\frac{b_{N/2+1,j}}{2}\bigg)U(x_{N/2+1},y_j)-\frac{f(x_{N/2+1},y_j)}{2}\bigg].
\end{multline*}
Inserting the expressions for $U(x_{N/2-2},y_j)$ and $U(x_{N/2+2},y_j)$ in (\ref{diseqn}) gives 
\begin{multline*}
  L_T^{N,N}U(x_{N/2},y_j) =\bigg[\frac{a_{N/2+1,j}}{2\epsilon^2}+\frac{\epsilon^2}{2h_1(\epsilon^2+h_1a_{N/2-1,j})}-\frac{1}{H_2}-\frac{3}{2h_1}\bigg]U(x_{N/2},y_j)\\
  \\
  -\bigg[\frac{a_{N/2+1,j}}{2\epsilon^2}+\frac{H_2b_{N/2+1,j}}{4\epsilon^2}-\frac{1}{H_2}\bigg]U(x_{N/2+1},y_j)\\
  -\bigg[\frac{h_1}{2(\epsilon^2+h_1a_{N/2-1,j})}\bigg(\frac{2\epsilon^2}{h_1^2}+\frac{a_{N/2-1,j}}{h_1}+\frac{b_{N/2-1,j}}{2}\bigg)-\frac{2}{h_1}\bigg]U(x_{N/2-1},y_j)\\
  +\frac{H_2}{4\epsilon^2}f(x_{N/2+1},y_j)+\frac{h_1}{4(\epsilon^2+h_1a_{N/2-1,j})}f(x_{N/2-1},y_j).
\end{multline*}
We define the finite difference method
\begin{align}\label{diseqn1}
  L_{\epsilon}^{N,N}U(x_i,y_j)=
  \begin{cases}
    L_{u,\epsilon}^{N,N}U(x_i,y_j) =f(x_i,y_j),  & \forall (x_i,y_j)\in \Omega^{N,N},\\
    L_{T,\epsilon}^{N,N}U(x_{N/2},y_j)=-\frac{H_2}{4\epsilon^2}f(x_{N/2+1},y_j) \\
    \mbox{ \qquad \qquad } - \frac{h_1}{4(\epsilon^2+h_1a_{N/2-1,j})}f(x_{N/2-1},y_j), & 1\leq j\leq N,\\
    L_{m,\epsilon}^{N,N}{U}(x_i,y_j) =\hat{f}(x_i,y_j),  &  1\leq i\leq N,\, j=N/2,\\
    U(x_i,y_j)=q_1(y_j), & (x_i,y_j)\in \varrho_1^{N,N},\\
    U(x_i,y_j)=q_2(x_i), & (x_i,y_j)\in \varrho_2^{N,N},\\
    U(x_i,y_j)=q_3(y_j), & (x_i,y_j)\in \varrho_3^{N,N},\\ 
    U(x_i,y_j)=q_4(x_i), & (x_i,y_j)\in \varrho_4^{N,N},
  \end{cases}
\end{align}
where  $\varrho_{k}^{N,N}$ is the restriction of the mesh
$\overline{\Omega}^{N,N}$ to the boundary $\varrho_k$ as defined in
\eqref{eq:boundaries}.

\begin{lemma}{(Discrete maximum principle):}\label{dismaxprinciple}
  Let $L_{\epsilon}^{N,N}$ be the discrete operator given in
  (\ref{diseqn1}).
  Suppose that $\phi(x_i,y_j)\geq 0$ on $\varrho^{N,N}$, and
  $L_{u,\epsilon}^{N,N}\phi(x_i,y_j)\geq 0$, $\forall (x_i,y_j)\in
  \Omega^{N,N}$
  with
  $L^{N,N}_{T,\epsilon}{\phi}(x_{N/2},y_j)\geq 0$, and
  $L^{N,N}_{m,\epsilon}{\phi}(x_i,y_{N/2})\geq
  0$.
  Then $\phi(x_i,y_j)\geq 0$, for all $(x_i,y_j)\in \overline{\Omega}^{N,N}$. 
\end{lemma}
\begin{proof}
  Consider the function $\omega$ on $\bar{\Omega}^{N,N}$ defined as  $\phi(x_i,y_j)=\omega(x_i,y_j)\psi(x_i,y_j)$ where the function
  \[\psi(x_i,y_j)=\exp\bigg(\frac{\alpha(x_i-d_1)}{2\epsilon}+\frac{\beta(y_j-d_2)}{2\epsilon}\bigg),\quad (x_i,y_j)\in \bar{\Omega}^{N,N},\]
  and $\alpha>0$ and $\beta>0$ are some constants. Let
  \[
    \omega(x_p,y_q)=\min\limits_{(x_i,y_j)\in
      \bar{\Omega}^{N,N}}\{\omega(x_i,y_j)\}.\]
  If $\omega(x_p,y_q)\geq
  0,$ there is nothing to prove.
  Suppose instead that $\omega(x_p,y_q)<0$. 
  Then
  \[
    D_x^-\omega(x_p,y_q)\leq 0\leq D_x^+\omega(x_p,y_q),\quad D_y^-\omega(x_p,y_q)\leq 0\leq D_y^+\omega(x_p,y_q),\]
  and consequently $L_{u,\epsilon}^{N,N}\phi(x_p,y_q) =
  -\epsilon^2(\delta_{xx}^2+\delta_{yy}^2)\omega(x_p,y_q)+a_{i,j}D_x^-\omega(x_p,y_q)+b_{i,j}\omega(x_p,y_q)
  < 0$, which contradicts the Lemma's premise if $(x_p,y_q)\in
  \Omega^{N,N}$. Therefore, the only possibilities are that
  either $x_p=x_{N/2}$ or $y_q=y_{N/2}$. In the first instance, 
  \[			
    D_x^-\omega(x_{N/2},y_q)\leq 0\leq
    D_x^+\omega(x_{N/2},y_q),\quad D_y^-\omega(x_p,y_{N/2})\leq
    0\leq D_y^+\omega(x_p,y_{N/2}),
  \]
  and so $\omega(x_{N/2-1},y_q)=\omega(x_{N/2},y_q)=\omega(x_{N/2+1},y_q)<0$.
  Then, $L_{T,\epsilon}^{N,N}{\phi}(x_{N/2},y_q)<0$, which leads to
  the expected contradiction. Similar reasoning heads to a
  contradiction when  $y_q=y_{N/2}$.
\end{proof}
		
\begin{theorem}{(Discrete stability result:)}\label{disstaresult}
  If $U(x_i,y_j)$ is the discrete solution of (\ref{diseqn}), then 
  $$\big|\big|U(x_i,y_j)\big|\big|_{\bar{\Omega}^{N,N}}\leq \frac{1}{\alpha}||f||_{\Omega^{N,N}}+\max\limits_{\varrho^{N,N}}\bigg\{||u||\bigg\}, \quad \forall (x_i,y_j)\in \bar{\Omega}^{N,N}.$$
\end{theorem}
\begin{proof}
  Define the barrier function $\phi^{\pm}(x_i,y_j)$ as
  follows:
  \begin{align*}
    \phi^{\pm}(x_i,y_j)=	\begin{cases}
      M+\frac{||f||_{\Omega^{N,N}}}{\alpha}\bigg(1+\frac{x_i}{d_1}+\frac{y_j}{d_2}\bigg)\pm u(x_i,y_j),  &  (x_i,y_j)\in  {\Omega}^{N,N}_1,\\
      M+\frac{||f||_{\Omega^{N,N}}}{\alpha}\bigg(1+\frac{(1-x_i)}{(1-d_1)}+\frac{y_j}{d_2}\bigg)\pm u(x_i,y_j),  &  (x_i,y_j)\in  {\Omega}^{N,N}_2,\\
      M+\frac{||f||_{\Omega^{N,N}}}{\alpha}\bigg(1+\frac{x_i}{d_1}+\frac{(1-y_j)}{(1-d_2)}\bigg)\pm u(x_i,y_j),  & (x_i,y_j)\in  {\Omega}^{N,N}_3,\\
      M+\frac{||f||_{\Omega^{N,N}}}{\alpha}\bigg(1+\frac{(1-x_i)}{(1-d_1)}+\frac{(1-y_j)}{(1-d_2)}\bigg)\pm u(x_i,y_j),  &  (x_i,y_j)\in  {\Omega}^{N,N}_4.
    \end{cases}
  \end{align*}
  where, $ M = \max\limits_{\varrho^{N,N}}\{ || u ||\}$.			
  Then, clearly $\phi^{\pm}(x_i,y_j)\geq 0$ for each $(x_i,y_j)\in \varrho^{N,N}$. For each $(x_i,y_j)\in \Omega^{N,N},$ we have 
  \[L^{N,N}_{u,\epsilon}\phi^{\pm}(x_i,y_j)\geq 0.\]
  At the points $(x_{N/2},y_k)$ where $y_k \in \overline{\Omega}_y^N$, 
  \[
    L^{N,N}_{T,\epsilon}{\phi^\pm}(x_{N/2},y_j)\geq 0.
  \]
  Similarly, we can prove $(x_k,y_{N/2}) \in (x_i,y_{N/2})$ such that
  $L^{N,N}_{m,\epsilon}{{\phi}^\pm}(x_i,y_{N/2})\geq 0.$
  It is apparent from Lemma \ref{dismaxprinciple} that $\phi^{\pm}(x_{i},y_{j}) \geq 0 \quad \forall (x_{i},y_{j}) \in \bar{\Omega}^{N,N}$, which leads to obtaining our desire bound on $U(x_{i},y_{j})$.
\end{proof}

\section{Error analysis}\label{sec:analysis}

Notwithstanding, to find the appropriate bounds for this error, we must decompose the discrete solution in the same way as the exact solution is decomposed.  The numerical approximation can be expressed as,
\begin{multline*}
  U=\sum_{k=1}^{4}V_{k}+\sum_{i=1}^{2}({W}_{B_i}+{W}_{T_i}+{W}_{R_i})+\sum_{j=1}^{6}{W}_{D_j}\\
   +({W}_{B_2R_1}+{W}_{D_2R_1}+{W}_{D_4R_2}+{W}_{T_2R_2}+{W}_{D_5B_1}+{W}_{D_5D_1}+{W}_{D_6D_3}+{W}_{D_6T_1}).
 \end{multline*}
 Here, $V_k$ is the discrete  component; ${W}_{R_i},\,
 {W}_{B_i},\,{W}_{T_i},\,{W}_{D_j}$ are the discrete singular
 components and
 ${W}_{B_2R_1},\,{W}_{D_2R_1},\,{W}_{D_4R_2},\,{W}_{T_2R_2},\,
 {W}_{D_5B_1},\, {W}_{D_5D_1},\, {W}_{D_6D_3},\, {W}_{D_6T_1}$,
   all are associated with the terms from Section~\ref{sec:a priori}
 (see Figure~\ref{fig:Omega}).
 The various components satisfy the following equations.
 \begin{subequations}
   \begin{align}\label{eqn}
     L^{N,N} _{\epsilon}V_k(x_i,y_j) = f(x_i,y_j),  \quad \forall (x_i,y_j) \in  \Omega_{k}^{N,N}, \quad k=1,2,3,4,\\
     V_k(x_i,y_j) = v_k(x_i,y_j), \quad \forall (x_i,y_j)  \in \varrho^{N,N},\\
     [V_k](d_1,y_j) = [v_k](d_1,y_j) ; \quad  [V_k](x_i,d_2) = [v_k](x_i,d_2),\\
     \label{eqn1}
     [(V_k)_x](d_1,y_j) = [(v_k)_x](d_1,y_j) ; \quad [(V_k)_y](x_i,d_2) = [(v_k)_y](x_i,d_2).
   \end{align}			
 \end{subequations}
 \begin{subequations}
   \begin{align}\label{layereqn}
     L^{N,N}_{\epsilon}{W}_{B_l}(x_i,y_j) = 0,  \quad \forall (x_i,y_j) \in  \Omega_{k}^{N,N}, \quad l=1,2,\,k=1,2,\\
     {W}_{B_l}(x_i,y_j) = {w}_{B_l}(x_i,y_j), \quad \forall (x_i,y_j)  \in \varrho^{N,N},\\
     [{W}_{B_l}](d_1,y_j) = [{w}_{B_l}](d_1,y_j) ; \quad
     [{W}_{B_l}](x_i,d_2) = [{w}_{B_l}](x_i,d_2),\\
     [({W}_{B_l})_x](d_1,y_j) = [({w}_{B_l})_x](d_1,y_j) ; \quad
     [({W}_{B_l})_y](x_i,d_2) = [({w}_{B_l})_y](x_i,d_2).
   \end{align}	
 \end{subequations}
		\begin{subequations}
			\begin{align}\label{layereqn1}
				L^{N,N}_{\epsilon}{W}_{R_l}(x_i,y_j) = 0,  \quad \forall (x_i,y_j) \in  \Omega_{k}^{N,N}, \quad l=1,2,\,k=2,4,\\
				{W}_{R_l}(x_i,y_j) = {w}_{R_l}(x_i,y_j), \quad \forall (x_i,y_j)  \in \varrho^{N,N},\\
				[{W}_{R_l}](d_1,y_j) = [{w}_{R_l}](d_1,y_j) ; \quad
				[{W}_{R_l}](x_i,d_2) = [{w}_{R_l}](x_i,d_2),\\
				[({W}_{R_l})_x](d_1,y_j) = [({w}_{R_l})_x](d_1,y_j) ; \quad
				[({W}_{R_l})_y](x_i,d_2) = [({w}_{R_l})_y](x_i,d_2).
			\end{align}
		\end{subequations}
	
	\begin{subequations}
		\begin{align}\label{toplayereqn1}
			L^{N,N}_{\epsilon}{W}_{T_l}(x_i,y_j) = 0,  \quad \forall (x_i,y_j) \in  \Omega_{k}^{N,N}, \quad l=1,2,\,k=3,4,\\
			{W}_{T_l}(x_i,y_j) = {w}_{T_l}(x_i,y_j), \quad \forall (x_i,y_j)  \in \varrho^{N,N},\\
			[{W}_{T_l}](d_1,y_j) = [{w}_{T_l}](d_1,y_j) ; \quad
			[{W}_{T_l}](x_i,d_2) = [{w}_{T_l}](x_i,d_2),\\
			[({W}_{T_l})_x](d_1,y_j) = [({w}_{T_l})_x](d_1,y_j) ; \quad
			[({W}_{T_l})_y](x_i,d_2) = [({w}_{T_l})_y](x_i,d_2).
		\end{align}
	\end{subequations}

		\begin{subequations}
		\begin{align}\label{intlayereqn1}
			L^{N,N}_{\epsilon}{W}_{D_l}(x_i,y_j) = 0,  \quad \forall (x_i,y_j) \in  \Omega_k^{N,N}, \quad l=1,2,\dots,6,\, k=1,2,3,4,\\
			{W}_{D_l}(x_i,y_j) = {w}_{D_l}(x_i,y_j), \quad \forall (x_i,y_j)  \in \varrho^{N,N},\\
			[{W}_{D_l}](d_1,y_j) = [{w}_{D_l}](d_1,y_j) ; \quad
			[{W}_{D_l}](x_i,d_2) = [{w}_{D_l}](x_i,d_2),\\
			[({W}_{D_l})_x](d_1,y_j) = [({w}_{D_l})_x](d_1,y_j) ; \quad
			[({W}_{D_l})_y](x_i,d_2) = [({w}_{D_l})_y](x_i,d_2).
		\end{align}
	\end{subequations}
		\begin{subequations}
			\begin{align}\label{cornereqn}
				L^{N,N}_{\epsilon} {W}_{B_2R_1}(x_i,y_j) = 0,  \quad \forall (x_i,y_j) \in  \Omega_{2}^{N,N},\\
				{W}_{B_2R_1}(x_i,y_j) = {w}_{B_2R_1}(x_i,y_j), \quad \forall (x_i,y_j)  \in \varrho^{N,N},\\
				[{W}_{B_2R_1}](d_1,y_j) = [{w}_{B_2R_1}](d_1,y_j) ; \quad
				[{W}_{B_2R_1}](x_i,d_2) = [{w}_{B_2R_1}](x_i,d_2),\\
				[({W}_{B_2R_1})_x](d_1,y_j) = [({w}_{B_2R_1})_x](d_1,y_j) ; \quad
				[({W}_{B_2R_1})_y](x_i,d_2) = [({w}_{B_2R_1})_y](x_i,d_2).
			\end{align}
		\end{subequations}
	
		\begin{lemma}\label{regdiscrete}
		If $v_k$ and $V_k$ are the solutions of (\ref{regresult}) and (\ref{eqn}), respectively, then, for $k=1,2,3,4,$
		\[\vert V_{k}(x_i,y_j)-v_{k}(x_i,y_j)\vert\leq CN^{-1},\, \textrm{for}\,\, (x_i,y_j)\in \Omega_{k}^{N,N}.\] 
	\end{lemma}
	\begin{proof}
		Given that, in \autoref{disstaresult},  we have demonstrated
		$\epsilon$-robust stability of the discrete operator, the numerical
		analysis of the method depends on  the local truncation error.
		Utilising standard truncation error bounds for the regular components \eqref{eqn} using the result
		(\ref{regularderieqn}) gives the following estimates:
		\begin{equation}\label{regularbound}
		  \vert L_{\epsilon}^{N,N} (V_k-v_k) (x_i,y_j) \vert \leq
		C N^{-1}\bigg( 
		  \epsilon^2\bigg\| \dfrac{\partial^3 v_k}{\partial x^3}
		  \bigg\| + \bigg\| \dfrac{\partial^2 v_k}{\partial x^2}
		  \bigg\| +
		  \epsilon^2\bigg\| \dfrac{\partial^3 v_k}{\partial y^3}
		  \bigg\|\bigg)\leq CN^{-1}.
		\end{equation}
		Using the discrete maximum principle we can obtain
		\begin{equation}\label{regularesult}
			\vert V_k(x_i,y_j) - v_k(x_i,y_j)\vert \leq CN^{-1}, \quad k=1,2,3,4,\quad \textrm{for}\,(x_i,y_j)\in\Omega_{k}^{N,N}.
		\end{equation}
			\end{proof}

To establish  $\epsilon$-uniform bounds on the truncation errors for
the components associated with  edge and corner functions,
we employ barrier functions that are somewhat standard in the field:
\[
  \begin{cases}
    G_{{w}_{B_1};j} =
    \prod\limits_{s=1}^j(1+k_s{\alpha/\epsilon})^{-1}, &  0 < j \leq
    N/8,
    \quad 0 < i \leq N/2, \\
    G_{{w}_{B_2};j} =
    \prod\limits_{s=1}^j(1+k_s{\alpha/\epsilon})^{-1}, &  0 < j \leq
    N/8, \quad
    
     N/2+1 \leq i \leq N,\
   \end{cases}
 \]
\[
  \begin{cases}
    G_{{w}_{R_1};i} =
    \prod\limits_{s=i+1}^{N}(1+h_s{\alpha/\epsilon})^{-1},
    & N/2 \leq i \leq N, \quad 0 \leq j \leq N/2,\\
    G_{{w}_{R_2};i} =
    \prod\limits_{s=i+1}^{N}(1+h_s{\alpha/\epsilon})^{-1}, &
     N/2 \leq i \leq
     N, \quad N/2+1 \leq j \leq N,
   \end{cases}
\]
\[
  \begin{cases}
    G_{{w}_{T_1};j} =
    \prod\limits_{s=j+1}^{N}(1+k_s{\alpha/\epsilon})^{-1},
    &  0 \leq i \leq N/2, \quad N/2+1 \leq j \leq N,\\
    G_{{w}_{T_2};j} =
    \prod\limits_{s=j+1}^{N}(1+k_s{\alpha/\epsilon})^{-1},
    &  N/2+1 \leq i \leq N, \quad N/2+1\leq j \leq N,
  \end{cases}
\]
\[
  \begin{cases}
    G_{{w}_{D_1};j} =
    \prod\limits_{s=j+1}^{N/2}(1+k_s{\alpha/\epsilon})^{-1},
    &   0 \leq j \leq N/2, \quad 0 < i \leq N/2,\\
    G_{{w}_{D_2};j} =
    \prod\limits_{s=j+1}^{N/2}(1+k_s{\alpha/\epsilon})^{-1},
    &  0 \leq j \leq N/2, \quad N/2+1 \leq i \leq N,
  \end{cases}
\]			
\[
  \begin{cases}
  G_{{w}_{D_3};j} =
  \prod\limits_{s=N/2+1}^j(1+k_s{\alpha/\epsilon})^{-1},
  &   0 < j \leq 5N/8, \quad 0 < i \leq N/2,\\
  G_{{w}_{D_4};j} =
  \prod\limits_{s=N/2+1}^j(1+k_s{\alpha/\epsilon})^{-1}, &
  0 < j \leq 5N/8, \quad N/2+1\leq i\leq N.
\end{cases}
\]
These functions are first-order Taylor approximations of the exponential
functions associated with the singular components of the  problem
(\ref{eqncon}). For all $j$, we have 
\[
  \exp\bigg(-\frac{\beta}{\epsilon} y_j\bigg) = \prod\limits_{s=1}^j
  \exp\bigg(-\frac{\beta}{\epsilon} k_s\bigg) \leq G_{{w}_{B_1};j},
\]
\[
  \exp\bigg(-\frac{\beta}{\epsilon} (1-x_i)\bigg) = \prod\limits_{s=1}^i
  \exp\bigg(-\frac{\alpha}{\epsilon^2} h_s\bigg) \leq G_{{w}_{R_1};i}.
\]
For $1 \leq i \leq N/2$ and $N/8 \leq j \leq N/2$ we have		
\begin{equation}\label{eqnb1}
  G_{{w}_{B_1};j} \leq G_{{w}_{B_1};N/8} \leq C \exp\bigg(\sum_{s=1}^{N/8}\bigg(\frac{1}{2}\bigg(\frac{\beta k_s}{\epsilon}\bigg)^2-\frac{\beta k_s}{\epsilon}\bigg)\bigg)\leq C N^{-1}.
\end{equation}
For $N/2+1 \leq i \leq 3N/4$ and $1 \leq j \leq N/2$ we have
\begin{equation}\label{eqnb11}
  G_{{w}_{R_1};i} \leq G_{{w}_{R_1};3N/4} \leq C \exp\bigg(\sum_{s=3N/4+1}^{N}\bigg(\frac{1}{2}\bigg(\frac{\alpha h_s}{\epsilon^2}\bigg)^2-\frac{\alpha h_s}{\epsilon^2}\bigg)\bigg)\leq C N^{-1}.
\end{equation}
Similar bounds may be derived for the remaining edge functions.
		
\begin{lemma}\label{singularlemma}
  If ${w}_{B_l}$ and ${W}_{B_l}$ are the solutions of
  (\ref{singularequation111}) and (\ref{layereqn}) respectively then,
  for  $l=1,2$, and $k=1,2$, 
  \begin{equation}
    \label{eq:wbk error bound}
    \vert{w}_{B_l}(x_i,y_j) - {W}_{B_l}(x_i,y_j)\vert \leq CN^{-1} \ln
    N, \quad (x_i,y_j)\in \Omega_k^{N,N}.
  \end{equation}
\end{lemma}
		
\begin{proof}
  We consider only the case $\sigma_y < 0.125$ in detail;
    otherwise classical arguments apply since $\epsilon^{-2} \leq C\ln
    N$. Furthermore, to avoid repetition, we verify the result only
    for the layer function ${w}_{B_1}$. From (\ref{layereqn}) and  \autoref{decompostiontheorem} it follows that 
  \begin{equation}
    \vert{W}_{B_1}(x_i,y_j)\vert=\vert{w}_{B_1}(x_i,y_j)\vert\leq C e^{(-\beta/\epsilon)y_j}\leq C G_{{w}_{B_1},j},\quad (x_i,y_j)\in \varrho^{N,N}.
  \end{equation}
  Further, for all internal grid points
  $(x_i,y_j)\in \Omega_1^{N,N}$,
  from (\ref{layereqn}) and the discrete maximum principle, it follows that 
  \begin{equation}\label{eqnprev}
    \vert{W}_{B_1}(x_i,y_j)\vert\leq G_{{w}_{B_1},j}.
  \end{equation}
  Therefore, applying \autoref{decompostiontheorem} and (\ref{eqnprev}), we conclude that
  \[	\vert{w}_{B_1}(x_i,y_j) - {W}_{B_1}(x_i,y_j)\vert \leq \vert{w}_{B_1}(x_i,y_j)\vert + \vert{W}_{B_1}(x_i,y_j)\vert \leq C G_{{w}_{B_1},j}
  \]
  Therefore, from (\ref{eqnb1}), we have 
  \begin{equation}\label{seqn}
    \vert{w}_{B_1}(x_i,y_j) - {W}_{B_1}(x_i,y_j)\vert \leq C N^{-1}, \qquad (x_i,y_j)\in \bar{\Omega}_1^{N,N}\setminus {\Omega_1^*}^{N,N}.
  \end{equation}
  To establish similar bounds on the error in the region
  ${\Omega_1^*}^{N,N} = \big\{(x_i,y_j) \,|\, 0 < i < N/2, \, 0 < j
  < N/8 \big\}$, we continue in the following manner. Applying Taylor
  expansions, we get  
  \[
    \vert L_{\epsilon}^{N,N} ({W}_{B_1}(x_i,y_j)-{w}_{B_1}(x_i,y_j))\vert \leq
    CN^{-1}  \bigg(\epsilon^2 \bigg\| \dfrac{\partial^3 {w}_{B_1}}{\partial x^3} \bigg\| + \bigg\| \dfrac{\partial^2 {w}_{B_1}}{\partial x^2} \bigg\| \bigg)+C \epsilon^2 \sigma_y N^{-1} \bigg\| \dfrac{\partial^3 {w}_{B_1}}{\partial y^3} \bigg\|.
  \]
  From \autoref{decompostiontheorem}, it follows that
  \begin{align*}
    \vert L_{\epsilon}^{N,N}[{W}_{B_1}(x_i,y_j) - {w}_{B_1}(x_i,y_j)]\vert \leq
    CN^{-1}+CN^{-1}\ln N.
  \end{align*}
  Hence, using Lemma \ref{dismaxprinciple} applied only on $\bar{\Omega}_1^{*N,N}$, to get
  \begin{equation}\label{lasteqn}
    \vert{w}_{B_1}(x_i,y_j) - {W}_{B_1}(x_i,y_j)\vert \leq C N^{-1} \ln N, \quad (x_i,y_j) \in \bar{\Omega}_1^{*N,N}.
  \end{equation}
  The result follows from (\ref{seqn}) and (\ref{lasteqn}).  Apply the same reasoning for the boundary layer component $w_{B_2}$
  to obtain the result in \eqref{eq:wbk error bound} for $l=2$.
\end{proof}

The reasoning that led to proof of \autoref{singularlemma} applies directly to components of the layer terms associated with the top of the domain, and interior discontinuity, so we state the following lemma without proof.
\begin{lemma}\label{para_singularlemma}
At each grid points $(x_i,y_j)\in \Omega_k^{N,N},\,k=1,2,3,4$, the error in the approximation of  the layer terms associated with the top edge and inner characteristic layer components (as depicted in Figure~\ref{fig:Omega}) may be bounded as 
	\begin{align*}
		\vert{w}_{T_j}(x_i,y_j) - {W}_{T_j}(x_i,y_j)\vert &\leq C N^{-1} \ln N, \quad j=1,2,\\
			\vert{w}_{D_l}(x_i,y_j) - {W}_{D_l}(x_i,y_j)\vert &\leq C N^{-1} \ln N, \quad l=1,2,3,4.
	\end{align*}
      \end{lemma}

The arguments for the terms associated with the right-hand boundary of the domain are a little different, so we provide some details.
\begin{lemma}\label{singularlemma1}
  If ${w}_{R_m}$ and ${W}_{R_m}$ are the solutions of (\ref{singularequation11}) and (\ref{layereqn1}), respectively, then, for  $m=1,2,\,k=2,4,$
  \[\vert{w}_{R_m}(x_i,y_j) - {W}_{R_m}(x_i,y_j)\vert \leq CN^{-1} \ln^2 N, \quad (x_i,y_j)\in \Omega_k^{N,N}.\]
\end{lemma}
		
\begin{proof}
We detail the proof for the layer function ${w}_{R_1}$;
  the arguments for the other layer functions are essentially
  the same. Further, as before, we consider only the case
  where $\sigma_x < 0.25$.
		
  From (\ref{layereqn1}) and \autoref{decompostiontheorem} it follows that 
  \begin{equation}
    \vert{W}_{R_1}(x_i,y_j)\vert=\vert{w}_{R_1}(x_i,y_j)\vert\leq C e^{(-\alpha/\epsilon^2)(1-x_i)}\leq C G_{{w}_{R_1},i},\quad (x_i,y_j)\in \varrho^{N,N}.
  \end{equation}
  Further, for all internal grid points, from (\ref{layereqn}) and the discrete
  maximum principle, it follows that  
  \begin{equation}\label{eqnprev1}
    \vert{W}^{}_{R_1}(x_i,y_j)\vert\leq ^{}G_{{w}^{}_{R_1},i}.
  \end{equation}
  Therefore, applying \autoref{decompostiontheorem} and (\ref{eqnprev1}), we conclude that
  \[
    \vert{w}_{R_1}(x_i,y_j) - {W}_{R_1}(x_i,y_j)\vert \leq \vert{w}_{R_1}(x_i,y_j)\vert + \vert{W}_{R_1}(x_i,y_j)\vert \leq C G_{{w}_{R_1},i}
  \]
Then, from (\ref{eqnb11}), we have 
  \begin{equation}\label{seqn1}
    \vert{w}_{R_1}(x_i,y_j) - {W}_{R_1}(x_i,y_j)\vert \leq C N^{-1}, \qquad (x_i,y_j)\in \bar{\Omega}_2^{N,N}\backslash {\Omega_2^{*}}^{N,N}.
  \end{equation}
To demonstrate similar bounds for the error in the region ${\Omega_2^{*}}^{N,N} = \big\{(x_i,y_j) \,|\, 3N/4 < i < N, \, 0 < j <N/2 \big\}$, we continue in the following manner. Applying Taylor expansions, we get 
  \begin{align*}
    \vert L_{\epsilon}^{N,N} ({W}_{R_1}(x_i,y_j)-{w}_{R_1}(x_i,y_j))\vert \leq
    C  \bigg(\epsilon^2\sigma_x N^{-1} \bigg\| \dfrac{\partial^3 {w}_{R_1}}{\partial x^3} \bigg\| + \bigg\|\dfrac{\partial^2 {w}_{R_1}}{\partial x^2} \bigg\|\bigg)+C \epsilon^2N^{-1} \bigg\| \dfrac{\partial^3 {w}_{R_1}}{\partial y^3} \bigg\|.
  \end{align*}
  From \autoref{decompostiontheorem}, it follows that
  \begin{align*}
    \vert L_{\epsilon}^{N,N}[{W}_{R_1}(x_i,y_j) - {w}_{R_1}(x_i,y_j)]\vert \leq
    C\epsilon^{-2}N^{-1}\ln N.
  \end{align*}
Then one can employ the barrier function
$C \epsilon^{_2} N^{-1}\ln N(x_i-(1-\sigma_x))+CN^{-1}$,  
and Lemma~\ref{dismaxprinciple}  on $\bar{\Omega}_2^{*N,N}$, to obtain
\begin{equation}\label{lasteqn1}
  \vert{w}_{R_1}(x_i,y_j) - {W}_{R_1}(x_i,y_j)\vert \leq C N^{-1} \ln^2N, \quad (x_i,y_j) \in \bar{\Omega}_2^{*N,N}.
\end{equation}
The result follows from (\ref{seqn1}) and (\ref{lasteqn1}). 
\end{proof}

\begin{lemma}\label{lemma2}
  If ${w}^{}_{B_2 R_1}$ and ${W}_{B_2R_1}$ are the solutions of (\ref{cornerlayer}) and (\ref{cornereqn}), respectively, then 
  \begin{equation}\label{cornereqn1}
    \vert{w}_{B_2R_1}(x_i,y_j) - {W}_{B_2R_1}(x_,y_j)\vert \leq CN^{-1}\ln^2N, \qquad (x_i,y_j) \in \Omega^{N,N}.
  \end{equation}
\end{lemma}

\begin{proof}
   We provide the proof of (\ref{cornereqn1}) for the
  corner layer function ${W}_{B_2R_1}$ and in the value of $\sigma_x <
  1/4$.
  Using arguments like those in the proof of ~\cref{singularlemma}, we get
  \[
    \vert{W}_{B_2R_1}(x_i,y_j)\vert \leq C \min\{G_{{w}_{B_2};j},
    G_{{w}_{R_1};i}\}, \quad \text{if } (x_i,y_j) \in \varrho^{N,N},
  \]
  \[
    \vert{w}_{B_2R_1}(x_i,y_j) -  {W}_{B_2R_1}(x_i,y_j)\vert \leq
    C\min\{ G_{{w}_{B_2};j}, G_{{w}_{R_1};i}\}, \quad \text{if }
    (x_i,y_j) \in \Omega_2^{N,N}.\]
  Then applying (\ref{eqnb1}) and (\ref{eqnb11}), we deduce that
  \begin{align}\label{corner}
  	 \vert{w}_{B_2R_1}(x_i,y_j) - {W}_{B_2R_1}(x_i,y_j)\vert \leq C N^{-1}, \qquad (x_i,y_j) \in \bar{\Omega}_2^{N,N} \backslash {\Omega^*}^{N,N}_{1,2},
  \end{align}
  where, ${\Omega^*}^{N,N}_{1,2} = \{(x_i,y_j) \,|\,\, 3N/4 < i< N,\
  0<j < N/8\}$. In this region, the truncation error satisfies
  \begin{multline*}
    \big| L_{\epsilon}^{N,N} [{W}_{B_2R_1}(x_i,y_j) -
    {w}_{B_2R_1}(x_i,y_j)] \big| \\
    \leq  C\sigma_x N^{-1} \bigg(\epsilon^2 \bigg\| \dfrac{\partial^3
      {w}_{B_2R_1}}{\partial x^3} \bigg\|+ \bigg\| \dfrac{\partial^2
      {w}_{B_2R_1}}{\partial x^2} \bigg\| \bigg) 
      +C \sigma_y
    N^{-1}\epsilon^2 \bigg\|\dfrac{\partial^3 {w}_{B_2R_1}}{\partial
      y^3} \bigg\|    \\
    \leq  N^{-1} \epsilon^{-2} \ln N,
  \end{multline*}
  where we have applied \autoref{decompostiontheorem}, with the
  barrier function $C \epsilon^{-2}N^{-1}\ln  N(x_i-(1-\sigma_x))+CN^{-1}$.
  The discrete maximum principle, applied on
  $\bar{\Omega}^{*N,N}_{1,2}$,  leads to
  \begin{align}\label{corner1}
  	\vert{w}_{B_2R_1}(x_i,y_j) - {W}_{B_2R_1}(x_i,y_j)\vert \leq C
  	N^{-1} \ln^2 N.
  \end{align}
Then, the required result follows from the equations \eqref{corner} and \eqref{corner1}.
\end{proof}

As explained in the following lemma, analogous results hold for the remaining corner layer components
${w}_{D_2R_1}$, ${w}_{D_4R_2}$, ${w}_{T_2R_2}$, ${w}_{D_5B_1}$,
${w}_{D_5D_1}$, ${w}_{D_6D_3}$, and ${w}_{D_6T_1}$.
\begin{lemma}\label{corner_singularlemma}
At each grid points $(x_i,y_j)\in \Omega_k^{N,N},\,k=1,2,3,4$, the error in the approximations of the remaining corner layer components, as depicted in Figure~\ref{fig:Omega}, are bounded as 
\begin{align*}
	\vert{w'}(x_i,y_j) - {W'}(x_i,y_j)\vert \leq C N^{-1} \ln^2 N,
\end{align*}
where $W'={W}_{D_2R_1}+{W}_{D_4R_2}+{W}_{T_2R_2}+{W}_{D_5B_1}+{W}_{D_5D_1}+{W}_{D_6D_3}+{W}_{D_6T_1}$.
\end{lemma}

\begin{lemma}\label{dislayer1}
  If $W_{D_5}$ be the solution of problem (\ref{intlayereqn1}), then 
  \[\vert D_xW_{D_5}(d_1,y_j)\vert\leq C(1+\epsilon^{-2}N^{-1}).\]
\end{lemma}
\begin{proof}
  At $(x,y)=(d_1,y_j),$ let
  \[ D_x^{+}V_k(d_1,y_j)=D_x^{+}(V_k-v_k)(d_1,y_j)+D_x^{+}v_k(d_1,y_j).\]
  Since $\big\|\frac{\partial v_k}{\partial
    x}(d_1,y_j)\big\|_{\Omega_2}\leq C$, we have
  \[\vert
    D_x^{+}v_k(d_1,y_j)\vert\leq C\, \,\text{and}\,\, \vert
    D_x^{+}(V_k-v_k)(d_1,y_j)\vert=\bigg\vert\frac{(V_k-v_k)(d_1+H_2)-(V_k-v_k)(d_1)}{H_2}\bigg\vert\leq
    CN^{-1},\]
  by (\ref{regularesult}). Hence $\vert  D_x^{+}V_k(d_1,y_j)\vert\leq C(1+N^{-1}).$  
	
  Let $D_x^{-}V_k(d_1,y_j)=D_x^{-}(V_k-v_k)(d_1,y_j)+D_x^{-}v_k(d_1,y_j)$, and note that $\big\|\frac{\partial v_k}{\partial x}\big\|_{\Omega_1}\leq C$.  
As per \cite[Lemma~3.14]{farrell2000robust}, we write 
  $\vert\epsilon^2 D_x^{-}(V_k-v_k)(d_1,y_j)\vert\leq CN^{-1},$ which implies that 
  \[
  \vert D_x^{-}V_k(d_1,y_j)\vert\leq C(1+\epsilon^{-2} N^{-1}).
  \]
  Note that $\vert W_{R_1}(x_i,y_i)\vert\leq CN^{-1},\, (x_i,y_j)\in \Omega_2$, which implies that $\vert D_x^{+}W_{R_1}(d_1,y_j)\vert\leq C.$ Finally, note that $$D_x^{-}W_{R_1}(d_1,y_j)=D_x^{-}(W_{R_1}-w_{R_1})(d_1,y_j)+D_x^{-}w_{R_1}(d_1,y_j),$$
	and $\big\|\frac{\partial w_{R_1}}{\partial x}(d_1,y_j)\big\|_{\Omega_1}\leq C.$ Hence,
	$$\vert D_x^{-}W_{R_1}(d_1,y_j)\vert\leq \vert D_x^{-}(W_{R_1}-w_{R_1})(d_1,y_j)\vert+C.$$
	It is easy to show that
        \[
          \vert D_x^{-}(W_{R_1}-w_{R_1})(x_i,y_j)\vert\leq C,
        \]
        which implies that
        \[\vert D_x^{-}W_{R_1}(d_1,y_j)\vert\leq C.\]
\end{proof}
	Analogous results hold for the other interior layer component
        ${w}_{D_6}$, as exposited  in the following lemma. 

	\begin{lemma}\label{dis_singularlemma}
		At each grid points $(x_i,y_j)\in \Omega_2^{N,N}$, the error estimate of the another interior layer component $W_{D_6}$ which relate the
		regions focused in
		Figure~\ref{fig:Omega}, is as follows:
		\[\vert D_xW_{D_6}(d_1,y_j)\vert\leq C(1+\epsilon^{-2}N^{-1}).\]
	\end{lemma}

\begin{lemma}
  Let $W_{D_5}$ be the solution of problem (\ref{intlayereqn1})  then,
  uniformly in $\epsilon$,
  \[
    \vert W_{D_5}(x_i,y_j)\vert\leq C
  \epsilon^2\vert[D_xW_{D_5}(d_1,y_j)]\vert.\]
\end{lemma}
\begin{proof}
	Define the barrier function $\phi^\pm$ as
	\begin{align*}	\phi^\pm(x_i,y_j)=\frac{C\epsilon^2\vert{D_xW_{D_5}(d_1,y_j)}\vert}{\alpha}\begin{cases}
			\psi(x_i,y_j),\, (x_i,y_j)\in \Omega_1\\
			1,\, \qquad\quad (x_i,y_j)\in\Omega_2
		\end{cases}
		\pm W_{D_5},
	\end{align*}
	where $\psi$ is the solution of 
	\begin{align*}
		-\epsilon^2\delta_{xx}^2\psi(x_i,y_j)+\alpha D_x^{-}\psi(x_i,y_j)=0,\quad (x_i,y_j)\in \Omega_1,\\
		\psi(0,y_j)=0,\quad \psi(d_1,y_j)=1,\\
		D_x^{-}\psi(d_1,y_j),\,\, (x_i,y_j)\leq (d_1,y_j).
	\end{align*}
The proof follows from arguments in \cite[Chap. 3]{farrell2000robust}.
\end{proof}
	
\begin{lemma}\label{dislayereq}
  At each mesh point $(x_i,y_j)\in \bar{\Omega}^{N,N}$, the error of the singular component satisfies the estimates
  \begin{align*}
    \vert(W_{D_5}-w_{D_5})(x_i,y_j)\vert\leq CN^{-1}\ln N,\\
    \vert(W_{D_6}-w_{D_6})(x_i,y_j)\vert\leq CN^{-1}\ln N.
  \end{align*}
\end{lemma}
\begin{proof}
Firstly, we detail the proof for the interior layer component ${w}_{D_5}$;
the arguments for ${w}_{D_6}$ are essentially the same.  
Since $\big[\frac{\partial v_k}{\partial
    x}(d_1,y_j)\big]+\big[\frac{\partial w_{D_5}}{\partial
    x}(d_1,y_j)\big]=0,$ we have that
  \begin{multline*}
    \big[D_x(W_{D_5}-w_{D_5})(d_1,y_j)\big]=[D_xW_{D_5}(d_1,y_j)]-[D_xw_{D_5}(d_1,y_j)]\\
    = \bigg[\frac{\partial v_k}{\partial x}(d_1,y_j)\bigg]-[D_xV_k(d_1,y_j)]+\bigg[\frac{\partial w_{D_5}}{\partial x}(d_1,y_j)\bigg]-[D_xW_{D_5}(d_1,y_j)]-[D_xW_{R_1}(d_1,y_j)].
  \end{multline*}
  Note that
  \begin{multline*}
    \bigg[\frac{\partial v_{k}}{\partial
      x}(d_1,y_j)\bigg]-[D_xV_k(d_1,y_j)] \\=\frac{\partial v_{k}}{\partial
      x}({d_1}^+,y_j)-D_x^+v_k(d_1,y_j)+D_x^-v_k(d_1,y_j)-\frac{\partial
      v_{k}}{\partial x}({d_1}^-,y_j)+[D_x(V_k-v_k)(d_1,y_j)].
  \end{multline*}
  Also from the proof of \cref{dislayer1}, we have the bounds
  \[\vert[D_x(V_k-v_k)(d_1,y_j)]\vert\leq C\epsilon^{-2} N^{-1},\]
  and \[\bigg\vert\bigg[\frac{\partial v_k}{\partial x}(d_1,y_j)\bigg]-[D_xv_k(d_1,y_j)]\bigg\vert\leq CN^{-1}.\]
  Hence \[\bigg\vert\bigg[\frac{\partial v_k}{\partial x}(d_1,y_j)\bigg]-[D_xV_k(d_1,y_j)]\bigg\vert\leq C\epsilon^{-2} N^{-1}.\]
  Similarly, 
  \begin{multline*}
    \bigg\vert\bigg[\frac{\partial w_{D_5}}{\partial
      x}(d_1,y_j)\bigg]-[D_xw_{D_5}(d_1,y_j)]\bigg\vert \\
    \leq \bigg\vert D_x^+W_{D_5}(d_1,y_j)-\frac{\partial w_{D_5}}{\partial x}({d_1}^+,y_j)\bigg\vert+\bigg\vert D_x^-W_{D_5}(d_1,y_j)-\frac{\partial w_{D_5}}{\partial x}({d_1}^-,y_j)\bigg\vert\\
    \leq Ch_1\bigg\|\frac{\partial^2 w_{D_5}}{\partial x^2}({d_1}^-,y_j)\bigg|\bigg|+CH_2\bigg\|\frac{\partial^2 w_{D_5}}{\partial x^2}({d_1}^+,y_j)\bigg\|.
  \end{multline*}
  Using the bounds on the derivatives of $w_{D_5}$ given in
  \autoref{decompostiontheorem}, we have  
  \[\bigg|\bigg[\frac{\partial w_{D_5}}{\partial x}(d_1,y_j)\bigg]-[D_xw_{D_5}(d_1,y_j)]\bigg|\leq CN^{-1}\ln N.\]
  Also 
  \begin{align*}
    \vert[D_xw_{R_1}(d_1,y_j)]\vert\leq C(h_1+H_2)\bigg\vert\frac{\partial^2 w_{R_1}}{\partial x^2}(d_1+H_2,y_j)\bigg\vert\leq C(h_1+H_2)\epsilon^{-4}
                                 \leq CN^{-1}.
  \end{align*}
  By the analysis given in \cite[Section 3.5]{farrell2000robust}, we have 
  \[|[D_x(W_{R_1}-w_{R_1})(d_1,y_j)]|\leq C \epsilon^{-2}N^{-1}\ln N,\]
  which implies
  \[|[D_x(W_{D_5}-w_{D_5})(d_1,y_j)]|\leq C \epsilon^{-2}N^{-1}\ln N.\]
  Using standard truncation error analysis, and the bounds on the
  derivatives of $w_{D_5}$, it can be established that
    \[
  	\big| L_{\epsilon}^{N,N}(W_{D_5}-w_{D_5})(x_i,y_j) \big| \leq
  	  \begin{cases} 
   C\bigg\|\epsilon^2\frac{\partial^2 w_{D_5}}{\partial x^2}\bigg\|+C
  	\bigg\|\frac{\partial w_{D_5}}{\partial x}\bigg\|\leq CN^{-1} & (x_i,y_j)\in \left(0,d-\sigma_x\right]\times (0,d),\\
  	Ch_1\epsilon^{-2}\leq CN^{-1}\ln N & (x_i,y_j)\in   	(d-\sigma_x,d)\times (0,d),\\
  C\bigg\|\epsilon^2\dfrac{\partial^2 w_{D_5}}{\partial x^2}\bigg\|+C \bigg\|\dfrac{\partial w_{D_5}}{\partial x}\bigg\|\leq CN^{-1}
  &	(x_i,y_j)\in (d,1-\sigma_x)\times (0,d),\\ 
Ch_2\epsilon^{-2}\leq CN^{-1}\ln N & (x_i,y_j)\in \left[1-\sigma_x,1\right)\times (0,d)
  \end{cases}
  \]
Combining all of these gives
        \[L_\epsilon^{N,N}(W_{D_5}-w_{D_5})\leq CN^{-1}\ln N\, \quad  \text{and} \quad  \, \vert[D_x(W_{D_5}-w_{D_5})(d,y_j)]\vert\leq C\epsilon^{-2}N^{-1}\ln N.\]
        
        To conclude, choose the barrier function 
        \[\phi(x_i,y_j)=CN^{-1}\ln N
          \begin{cases}
            \psi(x_i,y_j),\quad (x_i,y_j)\in \Omega_1\\
            1   ,\,\,\quad\qquad \quad (x_i,y_j)\in \Omega_2
          \end{cases} + CN^{-1}\ln N (1-x_i),\]
      
        where $\psi$ is the solution of the problem 
        \begin{align*}
          -\epsilon^2\delta_{xx}^2\psi(x_i,y_j)+\alpha D_x^{-}\psi(x_i,y_j)=0,\quad (x_i,y_j)\in \Omega_1,\\
          \psi(0,y_j)=0,\quad \psi(d,y_j)=1.
        \end{align*}
    
        Applying discrete maximum principle, we get the required
        result. Following the same steps, we can established the
        desired result for $W_{D_6}$.
	\end{proof}

 The discrete solution $U(x_i,y_j)$ of (\ref{diseqn}) can be written as,
 \begin{align*}
   U(x_i,y_j)= 
   \begin{cases}
     (V_1 + {W}_{B_1} + {W}_{D_1}+{W}_{D_5}+{W}_{D_5B_1}+{W}_{D_5D_1} )(x_i,y_j)      & \forall (x_i,y_j) \in \Omega_1^{{N,N}},\\
     (V_2 + {W}_{B_2} + {W}_{R_1} + {W}_{D_2}+ {W}_{B_2R_1} + {W}_{D_2R_1})(x_i,y_j)  & \forall (x_i,y_j) \in \Omega_2^{{N,N}},\\
     (V_3  + {W}_{D_3}+{W}_{D_6}+ {W}_{T_1}+{W}_{D_6D_3}+{W}_{D_6T_1} )(x_i,y_j)      & \forall (x_i,y_j) \in \Omega_3^{{N,N}},\\
     (V_4 + {W}_{T_2} + {W}_{D_4} + {W}_{R_2} + {W}_{D_4R_2} + {W}_{T_2R_2}) (x_i,y_j) & \forall  (x_i,y_j) \in \Omega_4^{{N,N}},\\
     (V + W + W') (d^+,y_j) = (V + W + W')(d^-,y_j),\\
     (V + W + W') (x_i,d^+) = (V + W + W')(x_i,d^-).
   \end{cases}
 \end{align*}
 where,
\[
   V = \sum\limits_{i=1}^{4}V_{i}, \qquad
   W =
   \sum\limits_{j=1}^{2}({W}_{B_J}+{W}_{T_j}+{W}_{R_j})+\sum\limits_{k=1}^{6}{W}_{D_k},\]
 and
 \[
 W'
 ={W}_{B_2R_1}+{W}_{D_2R_1}+{W}_{D_4R_2}+{W}_{T_2R_2}+{W}_{D_5B_1}+{W}_{D_5D_1}+{W}_{D_6D_3}+{W}_{D_6T_1}.
 \] 
 
 \begin{theorem}\label{thm:convergence}
   Let $u$ be the solution to problem \eqref{eqncon} and $U$ be the
   numerical solution of (\ref{diseqn}) defined on the
   piecewise-uniform Shishkin mesh.  Then the error at the mesh points
   satisfies 
   \[
     \|{U}(x_i,y_j)-u(x_i,y_j)\| \leq C N^{-1} \ln^2 N , \quad \forall (x_i,y_j) \in \bar{\Omega}^{N,N}.
   \]
 \end{theorem}
 \begin{proof}
   Given the  bounds in (\ref{regularesult}) and Lemmas
   \ref{singularlemma}, \ref{para_singularlemma}, \ref{singularlemma1} \ref{lemma2}, \ref{corner_singularlemma}, and  \ref{dislayereq}, we deduce   
   the  uniform convergence result
\begin{align}\label{result_1}
  \vert U(x_i,y_j)-u(x_i,y_j)\vert \leq C N^{-1} \ln^2 N, \quad \forall (x_i,y_j) \in {\Omega}^{N,N}.
\end{align}
It remains to verify the result at the mesh points in
$\bar{\Omega}^{N,N} / {\Omega}^{N,N}$. Starting with $(x_i,y_{N/2})$,
denote ${e}(x_{i},y_{N/2}) =
{u}(x_{i},y_{N/2})-{U}({x}_{i},y_{N/2})$. Then
\begin{align*}
  L_{m,\epsilon}^{N,N} \hat{e}(x_{i},y_{N/2}) &= L_{m,\epsilon}^{N,N}\left(\hat{u}(x_{i},y_{N/2})-\hat{U}({x}_{i},y_{N/2})\right)\\[0.3cm]
                                              &= [f](x_{i},y_{N/2}) - L_{m,\epsilon}^{N,N} \hat{U}({x}_{i},y_{N/2}) \\[0.3cm]
                                              &=  (L_{\epsilon} - L_{m,\epsilon}^{N,N}) \hat{u}({x}_{i},y_{N/2})\\[0.3cm]
                                              &=  -\epsilon^2 \left( \left( \dfrac{\partial^{2}}{\partial x^{2}}-{\delta}_{xx}^{2}\right) + \left(\dfrac{\partial^{2}}{\partial y^{2}}-{\delta}_{yy}^{2}\right)\right) \hat{u}({x}_{i},y_{N/2})+\bigg(\frac{\partial}{\partial x}-D_x^- \bigg)\hat{u}({x}_{i},y_{N/2})\\[0.3cm]
                                              &\leq C{N^{-1}\ln N},
\end{align*}
where we have applied \autoref{decompostiontheorem}. The discrete
maximum principle now implies
\begin{align}\label{result_2}
\vert u(x_i,y_{N/2}) - U(x_i,y_{N/2})\vert \leq C N^{-1}\ln N.
\end{align}
One can easily show that the same bounds hold at $(x_{N/2},y_{j})=(d_1,y_j)$.
Hence, from the bounds (\ref{result_1}) and (\ref{result_2}), we get the required result.
\end{proof}

\section{Numerical Experiments}
\label{sec:numerics}  
We now present the results of numerical experiments that support the above theoretical results, based on two examples. 
Computations were performed in MATLAB (R2023b), and linear systems were resolved using its direct solver (which defaults to UMFPACK~\cite{Davi04}).

\begin{example}[Constant coefficients]\label{ex1}
Find $u=u(x,y)$ that satisfies
\[
 -\epsilon^2 \big(u_{xx}(x,y) + u_{yy}(x,y)\big) +2u_x(x,y)+25 u(x,y) = f(x,y), \quad \text{ on } \quad \Omega:=(0,1)^2,
 \]
 subject to the boundary conditions $u(x,y)|_{\partial \Omega} = 0$. For the right-hand side we take $d_1=d_2=1/2$, and
\[
f_1(x,y)=0.5, \qquad 
f_2(x,y)=0.6, \qquad 
f_3(x,y)=-0.6, \qquad 
f_4(x,y)=-0.5,
\]
where $f_k$ denotes the restriction of $f$ to $\Omega_k$ as defined in~\eqref{eq:subdomains}.
\end{example}
A plot of a typical solution to this problem is shown in \Cref{fig:examples1}. From this it is clear that there are boundary layers present near each of of the boundaries at $x=1$, $y=0$, and $y=1$, as well as the interior layers along the lines $x=1/2$ and $y=1/2$. The contour plot (\Cref{fig:eg1 contour}) highlights that there are two distinct layers regions in the $x$ coordinate, and four (much wider ones) in the $y$-coordinate.
\begin{figure}[ht]
\centering
\subfigure[Surface plot of solution]{\includegraphics[width=0.45\textwidth]{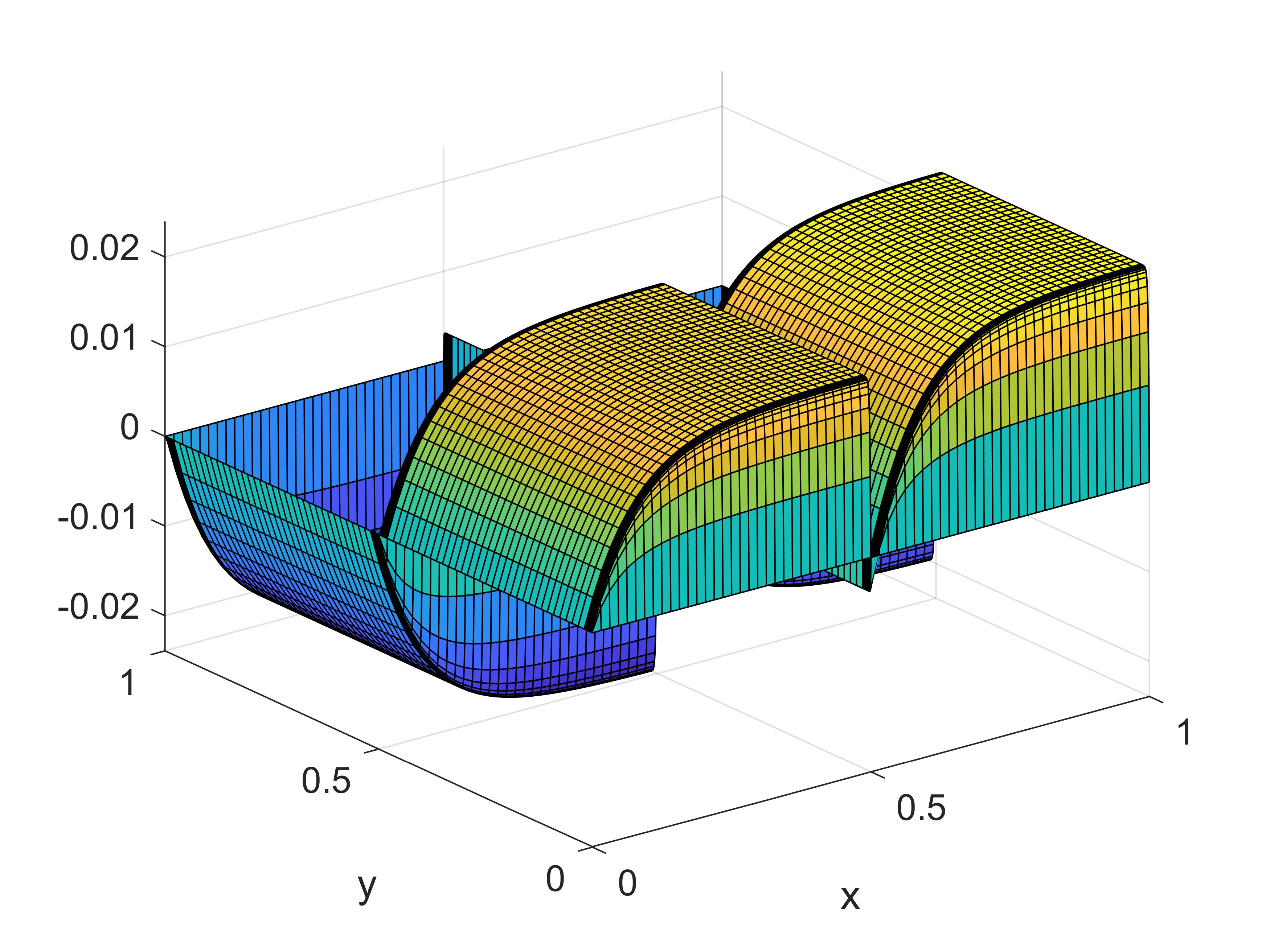}\label{fig:eg1 contour}}
\subfigure[Contour plot of solution]{\includegraphics[width=0.45\textwidth]{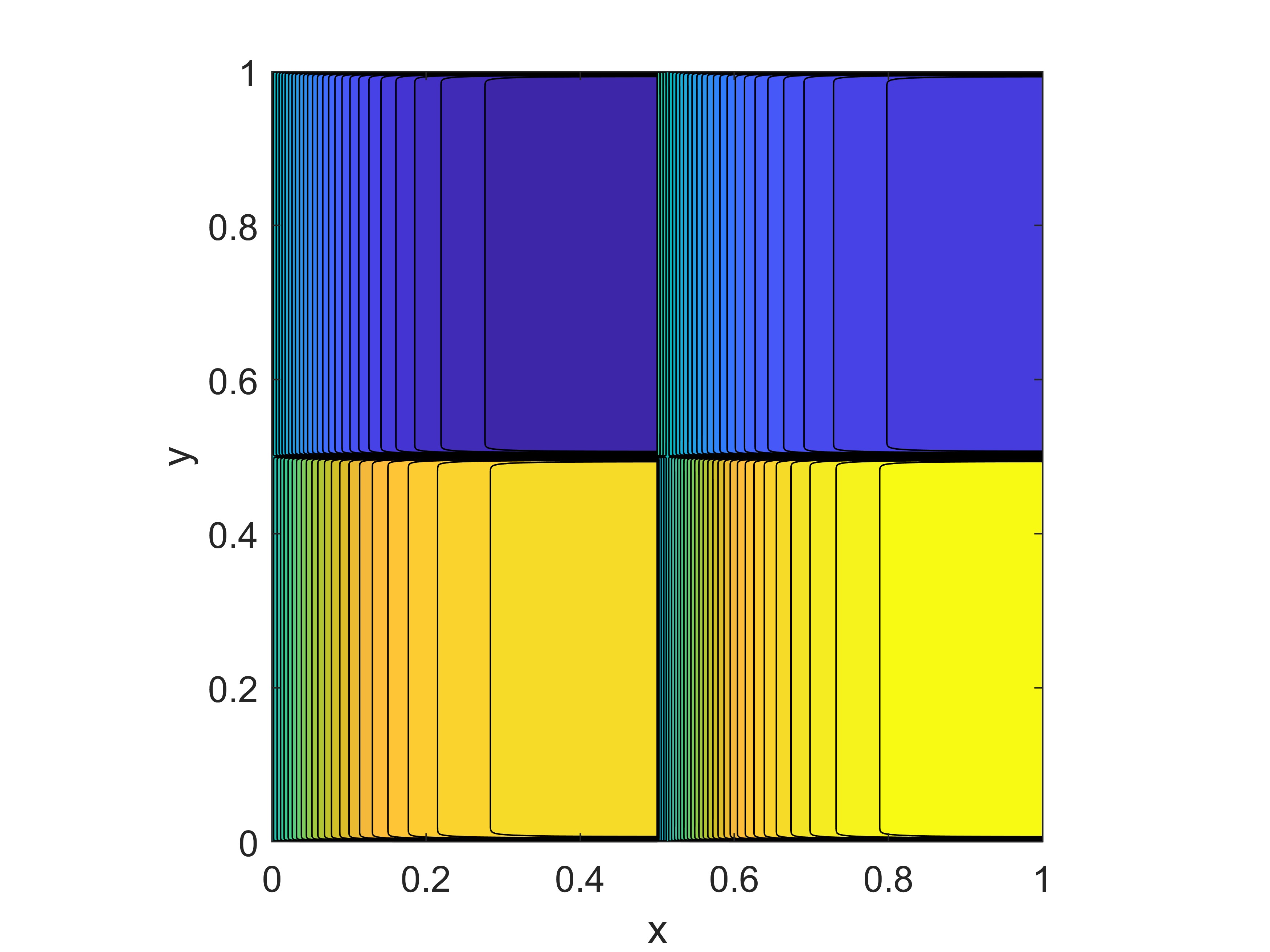}}
\caption{Surface and contour plots of numerical solution to \Cref{ex1} with $\epsilon = 10^{-2}$ and $N=128$} 
\label{fig:examples1}
\end{figure}

\begin{example}[Variable coefficients]\label{ex2}
We solve 
\[
-\epsilon^2 \big(u_{xx}(x,y) + u_{yy}(x,y)\big)+ (4+x)u_x(x,y)+(25+xy/2)u(x,y) = f(x,y),
\] 
again with homogeneous Dirichlet boundary conditions. This time we take the discontinuities at $d_1=0.4$, and $d_2=0.6$, and we set
\begin{align*}
f_1(x,y)&=(1+x+y), & 
f_2(x,y)&=-(1+x^2y^2), \\
f_3(x,y)&=-(1+xy), & 
f_4(x,y)&=(1+x+y).
\end{align*}
\end{example}
\cref{fig:examples2} show the surface and contour graphs of the numerical solution to \cref{ex1} for $\epsilon = 10^{-2}$ and $N=128$; note that for this example there are layers near $x=0.4$ and $y=0.6$.
\begin{figure}[ht]
\centering
\subfigure[Surface plot of solution]{\includegraphics[width=0.45\textwidth]{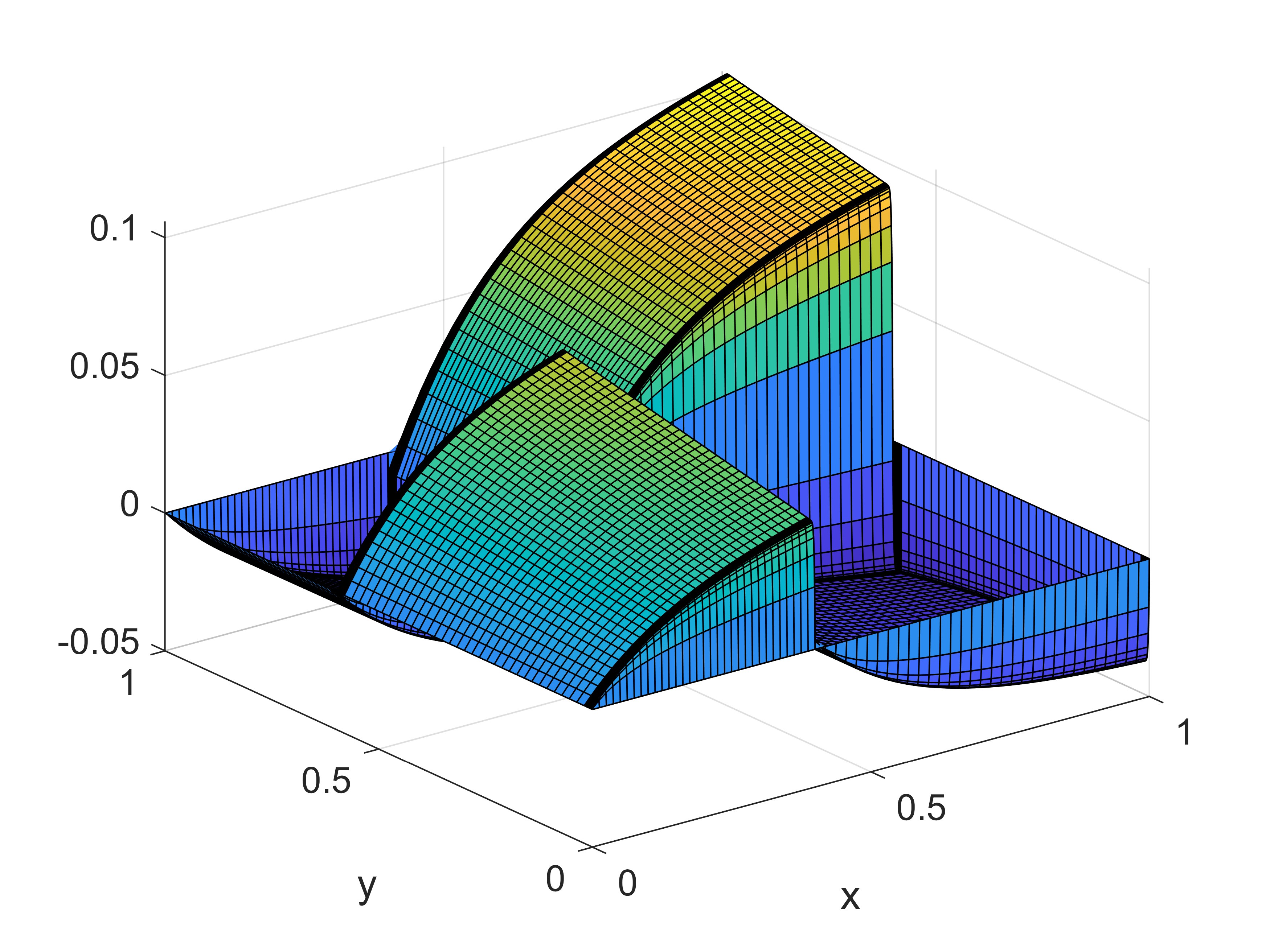}}
\subfigure[Contour plot of solution]{\includegraphics[width=0.45\textwidth]{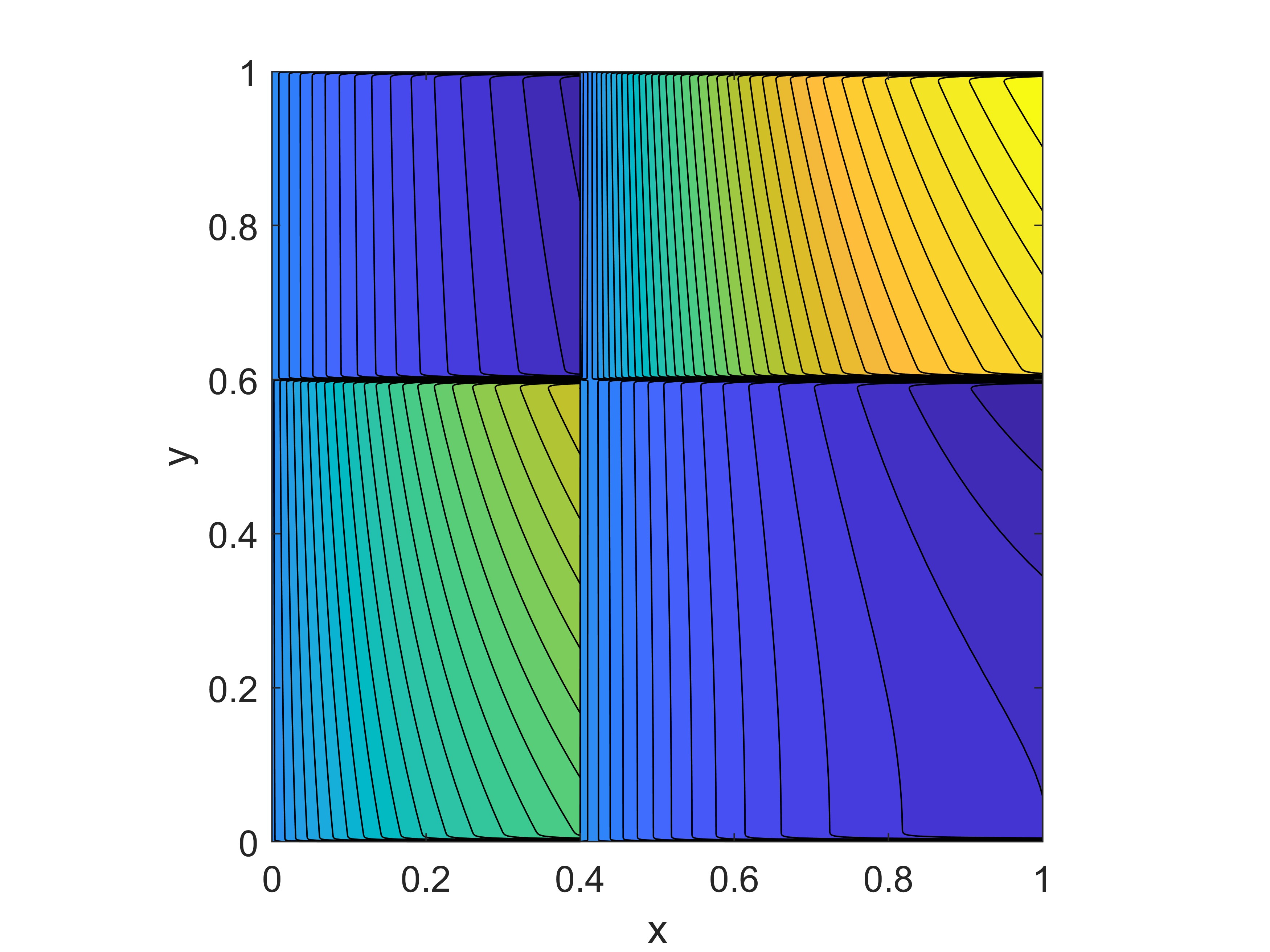}}
\caption{Surface graph and contour plots of numerical solution to \Cref{ex2}  with $\epsilon = 10^{-2}$ and $N=128$} 
\label{fig:examples2}
\end{figure}

For both these examples, the exact solution is unavailable.  Therefore, we use the double mesh principle to estimate the maximum point-wise error~\cite{farrell2000robust}.  That is, we set
\[ 
D^{N,N}_{\epsilon} = \displaystyle \max_{(x_i,y_j)\in \bar{\Omega}^{N,N}}  |\bar{U}^{2N,2N} (x_{2i},y_{2j})- U^{N,N} (x_i,y_j)|, \] 
where $\bar{U}^{2N,2N} (x_{2i},y_{2j})$ is the numerical solution on a $2N$ mesh intervals.  
The parameter uniform maximum point-wise error and estimated order of convergence are calculated, respectively, as 
\[ 
D^{N,N} = \max_{\epsilon} D^{N,N}_{\epsilon}, 
\quad \text{ and } \quad
E^{N,N}_{\epsilon} = \log_2\bigg(\frac{D^{N,N}}{D^{2N,2N}}\bigg).
\]
		
Results for \cref{ex1}, solved for a range of values of $N$ and $\epsilon$, are shown in Table~\ref{table1}. In these calculations, we have taken $\alpha=2$ and $\beta=5$.
As predicted by the theory, the error is independent of $\varepsilon$, and is first-order convergent in $N$.  This is a higher rate of convergence 
than predicted in \cref{thm:convergence}, and may be due to a simplification introduced by the problem featuring only constant coefficients.

\begin{table}[ht]
\caption{Maximum point-wise errors $D^{N,N}$ and orders of convergence $E^{N,N}$ for Example \ref{ex1}} 
\label{table1}
\centering
\begin{tabular}[h]{c|cccccc} 
\hline 
$N$/$\epsilon$ &  32 & 64 & 128 & 256 & 512  & 1024\\ [0.5ex] \hline  
1.00e-01 &  6.145e-03 &  3.221e-03 &  1.630e-03 &  8.111e-04 &  4.024e-04 &  2.022e-04 \\
1.00e-02 &  6.797e-03 &  3.667e-03 &  1.908e-03 &  9.740e-04 &  4.921e-04 &  2.480e-04 \\
1.00e-03 &  6.806e-03 &  3.672e-03 &  1.911e-03 &  9.758e-04 &  4.930e-04 &  2.485e-04 \\
1.00e-04 &  6.807e-03 &  3.672e-03 &  1.912e-03 &  9.758e-04 &  4.931e-04 &  2.485e-04 \\
1.00e-05 &  6.807e-03 &  3.672e-03 &  1.912e-03 &  9.758e-04 &  4.931e-04 &  2.485e-04 \\
1.00e-06 &  6.807e-03 &  3.672e-03 &  1.912e-03 &  9.758e-04 &  4.931e-04 &  2.485e-04 \\ \hline
$D^{N,N}$&  6.807e-03 &  3.672e-03 &  1.912e-03 &  9.758e-04 &  4.931e-04 &  2.485e-04 \\ \hline
$E^{N,N}$&      0.890 &      0.942 &      0.970 &      0.985 &      0.988 & ---
  \\\hline
\end{tabular}
\end{table}

Results for \Cref{ex2} are shown in \autoref{table2}. As before, we have taken $\alpha=2$ and $\beta=5$ when constructing the Shishkin mesh.
Again we observe $\epsilon$-uniform convergence, although in this case we note that the rate of convergence is almost first-order, which is in agreement with the theory, and not fully first-order, as was observed for \cref{ex1}. To investigate this further, in \cref{fig:errors} we plot the estimated errors for both examples, with $\epsilon=10^{-2}$ and $N=128$. In \cref{fig:error1}, we see that, for \Cref{ex1}, the maximum error appears 
to associated with the terms we denote $W_{D_5}$ and $W_{D_6}$, whereas for \cref{ex2}, the maximum error is associated with $W_{R_1}$ and 
$W_{R_2}$. We speculate that, in the case of constant coefficients, sharper bounds may be possible, but are of little interest.
It is of more significance, we believe, that \cref{fig:errors} demonstrates that the layers are well-resolved.

\begin{table}[ht]
\caption{Maximum point-wise errors $D^{N,N}$ and orders of convergence $E^{N,N}$ for Example \ref{ex2}}\label{table2}
\centering
\begin{tabular}[h]{c|cccccc} 
\hline 
    	$N$/$\epsilon$ &  32 & 64 & 128 & 256 & 512 & 1024 \\ [0.5ex] \hline  
$10^{-1}$ &  1.198e-02 &  7.393e-03 &  5.054e-03 &  3.310e-03 &  2.033e-03 &  1.200e-03 \\
$10^{-2}$ &  1.292e-02 &  7.748e-03 &  5.326e-03 &  3.482e-03 &  2.137e-03 &  1.259e-03 \\
$10^{-3}$ &  1.298e-02 &  7.770e-03 &  5.345e-03 &  3.496e-03 &  2.146e-03 &  1.264e-03 \\
$10^{-4}$ &  1.298e-02 &  7.772e-03 &  5.347e-03 &  3.498e-03 &  2.147e-03 &  1.265e-03 \\
$10^{-5}$ &  1.298e-02 &  7.772e-03 &  5.347e-03 &  3.498e-03 &  2.147e-03 &  1.265e-03 \\
$10^{-6}$ &  1.298e-02 &  7.771e-03 &  5.347e-03 &  3.498e-03 &  2.147e-03 &  1.265e-03   \\   \hline
$D^{N,N}$ &  1.298e-02 &  7.772e-03 &  5.347e-03 &  3.498e-03 &  2.147e-03 &  1.265e-03   \\   \hline
$E^{N,N}$ & 0.740 &      0.540 &      0.612 &      0.704 &      0.763 & --- \\ 	\hline
\end{tabular}
\end{table}
\begin{figure}[ht]
\centering
\subfigure[\Cref{ex1}]{\includegraphics[width=0.45\textwidth]{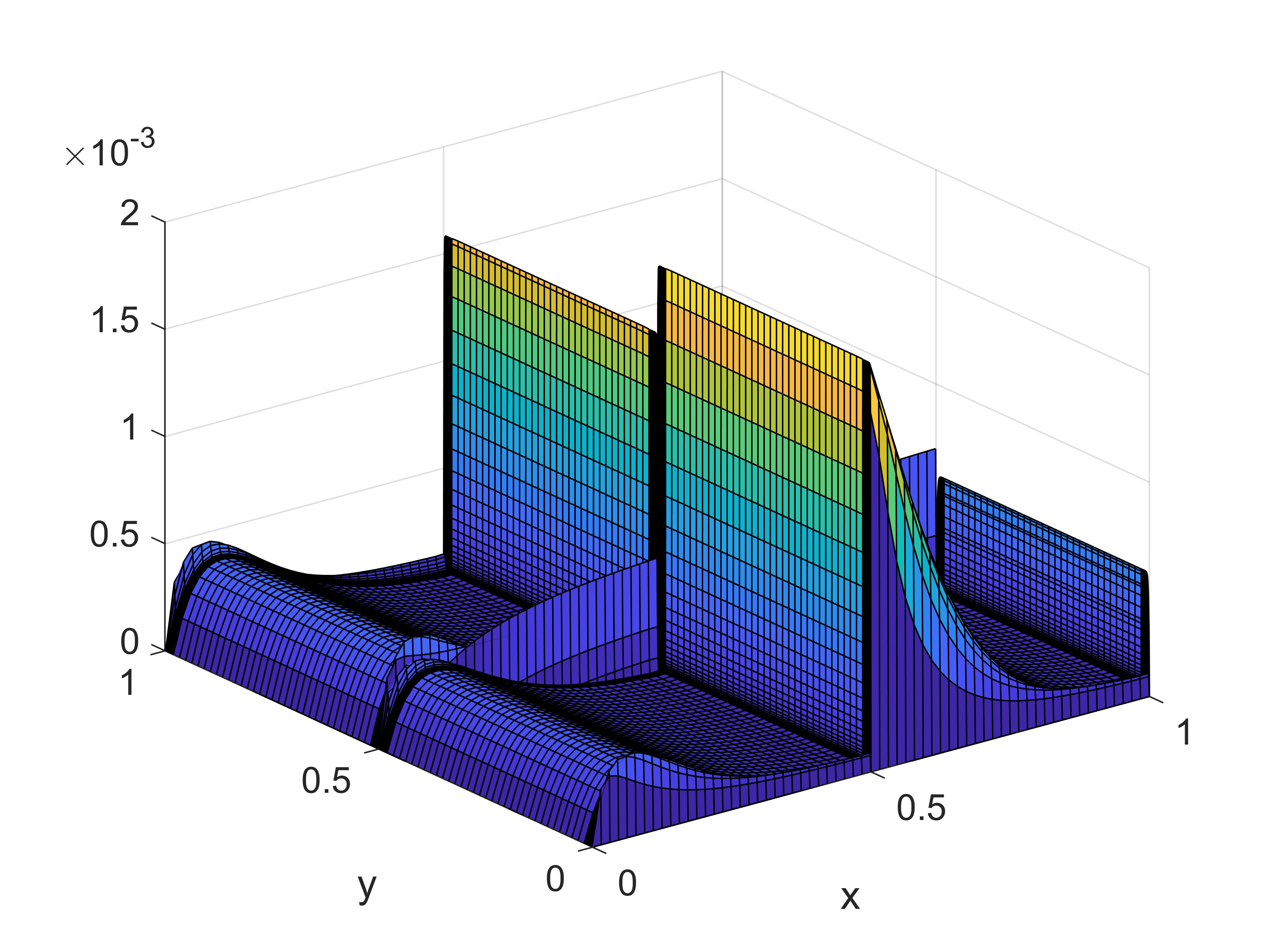}\label{fig:error1}}
\subfigure[\Cref{ex2}]{\includegraphics[width=0.45\textwidth]{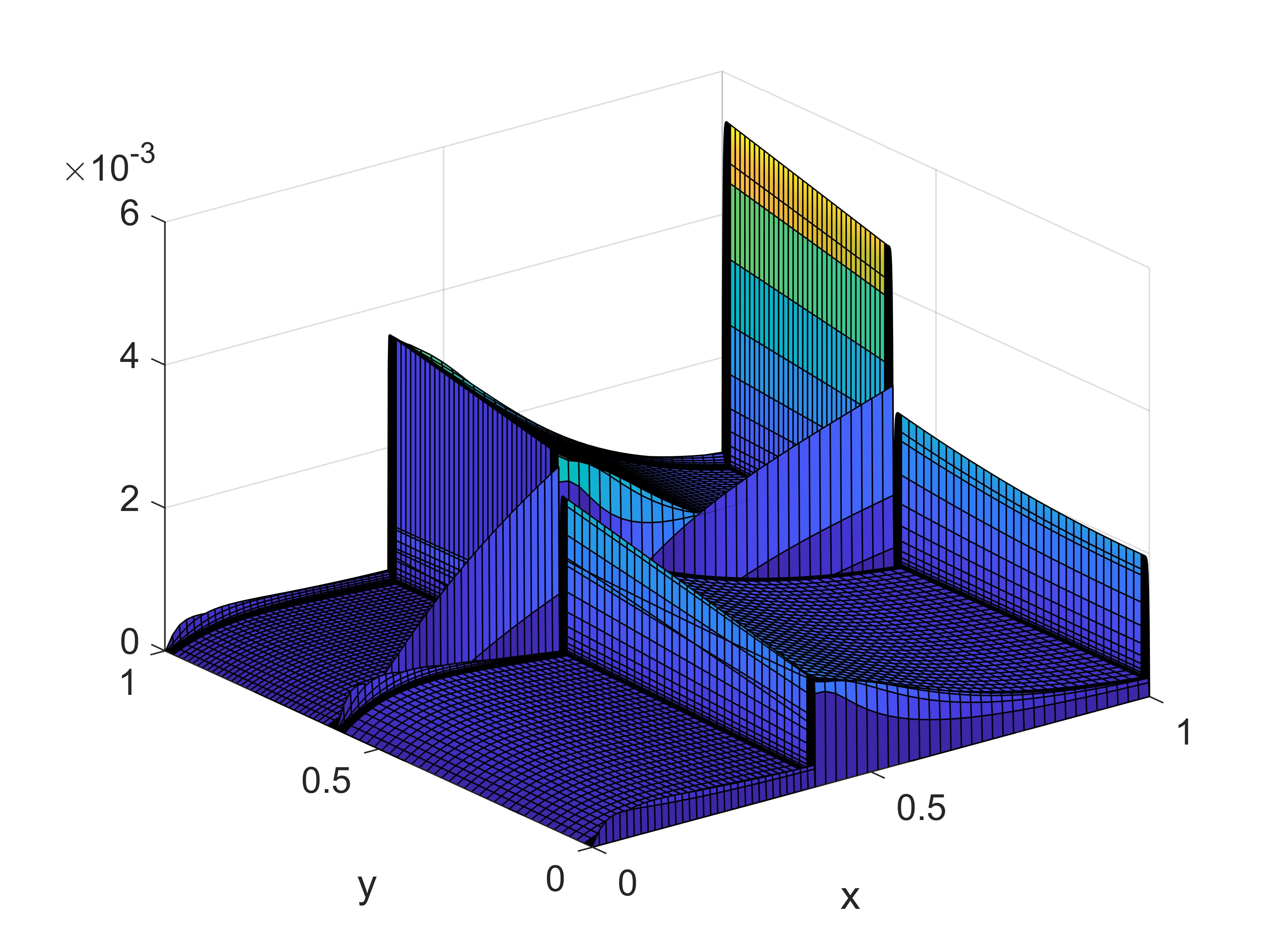}\label{fig:error2}}
\caption{Error graphs of numerical solution with $\epsilon = 10^{-2}$ and $N=128$} 
\label{fig:errors}
\label{fig:examples}
\end{figure}

\section{Conclusions}\label{sec:conclude}
We have provided the numerical analysis of a  parameter-uniform approximation technique  for robustly solving a two-dimensional singularly perturbed
convection-diffusion problem with boundary aligned flow, and with non-smooth data. 
Specifically, we have addressed the challenging issue of solving a problem where the solution exhibits two different type of interior layer (exponential and characteristic, which are of width 
$\mathcal{O}(\epsilon \ln \epsilon^{-1})$ and 
$\mathcal{O}(\epsilon^{-1/2} \ln \epsilon^{-1/2})$, respectively) 
and which interact with each other. This has involved obtaining sharp bounds for the solution decomposition, and the careful analysis of the method with respect to the many components involved. To achieve (almost) first-order convergence, we have use  a five-point difference scheme in combination with a mid-point scheme.  The resulting scheme is shown, in theory and practice, to be robust.

Of course there remain numerous open challenges, not least of which is when discontinuities in the boundary data, coupled with convective terms are not aligned with the domain's boundary, lead to interior layers that are no as easily resolved by tensor-product grid (see, \cite[Example 4.8]{Stynes18}, \cite[Example 3.2]{MaSt97} and more recently, \cite{FEM2023}). In addition, these scheme presented here is only (almost) first-order accurate, so it would be interesting to see how improve the approach so achieve either (fully) first-order, or higher-order convergence could be achieved.
This has been done for a related problem, but simpler problem featuring only exponential-type layers: see \cite{RaShVA23} for a fully first-order scheme on a graded mesh, and \cite{RaShPr2023} for a almost second-order scheme. There are subtle challenges involved in the error analysis for characteristic layers (see, e.g., \cite[\S5.3]{Stynes18}. Nonetheless, these studies  may prove a useful starting point in developing higher-order methods for this, more complicated problem.

Finally, we mentioned at the start of \S\ref{sec:numerics} that, for the examples reported, all the linear systems were solved using a direct solver. Since such solvers do not scale well, it would be preferable to use an properly preconditioned iterative solver. The recent work of Nhan et al.~\cite{MaMa22} describes an multigrid-based solver for a FDM discretization of \eqref{eq:full problem}, but with continuous data. That could be adapted to be applied to the method presented here.

\bibliographystyle{elsarticle-num} 
\bibliography{References_3}
		
\end{document}